\documentclass[11pt]{article}

\usepackage[T1]{fontenc}
\usepackage{lmodern}

\usepackage[width=17cm,height=23.5cm,centering]{geometry}
\usepackage{color,graphicx,subfig,enumitem}
\usepackage{amsmath,amssymb,amsthm,upgreek}
\usepackage{bm,mathrsfs} 
\DeclareMathAlphabet{\mathscrbf}{OMS}{mdugm}{b}{n}

\newtheorem{proposition}{Proposition}
\newtheorem{lemma}{Lemma}
\newtheorem{definition}{Definition}
\newtheorem{remark}{Remark}
\newtheorem*{remark*}{Remark}

\newcommand{\eps}{\varepsilon}
\newcommand{\sig}{\sigma}
\newcommand{\dd}{\,\text{d}}
\newcommand{\dc}{:}   
\newcommand{\TF}{\mathscr{F}}  
\newcommand{\tou}[1]{{\bm{#1}}}
\newcommand{\tod}[1]{{\bm{#1}}}
\newcommand{\toq}[1]{{\bm{#1}}}
\newcommand{\macr}[1]{\overline{#1}}
\newcommand{\eff}[1]{{#1}_{\operatorname{eff}}}
\newcommand{\moy}[1]{\big\langle #1 \big\rangle}
\newcommand{\psh}[2]{\big( #1,#2 \big)_{\!\mathscrbf{H}}}
\newcommand{\pse}[2]{\big( #1,#2 \big)_{\!\mathscrbf{H}_e}}
\newcommand{\pss}[2]{\big( #1,#2 \big)_{\!\mathscrbf{H}_s}}
\newcommand{\pseLz}[2]{\big( #1,#2 \big)_{\!\mathscrbf{H}_e,\toq{L}_0}}
\newcommand{\pssLz}[2]{\big( #1,#2 \big)_{\!\mathscrbf{H}_s,\toq{L}_0}}
\newcommand{\pseL}[2]{\big( #1,#2 \big)_{\!\mathscrbf{H}_e,\toq{L}}}
\newcommand{\pssL}[2]{\big( #1,#2 \big)_{\!\mathscrbf{H}_s,\toq{L}}}
\newcommand{\pdse}[2]{\big( #1,#2 \big)_{\!\mathscrbf{H}_s\times\mathscrbf{H}_e}}

\newcommand{\cS}[1]{{#1}_{\raisebox{-1pt}{\scriptsize$\mathscrbf{S}$}}}
\newcommand{\cSo}[1]{{#1}_{\mathscrbf{S}^\bot}}
\newcommand{\cEz}[1]{{#1}_{\raisebox{-1.5pt}{\scriptsize$\mathscrbf{E}_0$}}}
\newcommand{\cEzo}[1]{{#1}_{\mathscrbf{E}_0^\bot}}
\newcommand{\PEz}{\toq{P}_{\!\mathscrbf{E}_0}}
\newcommand{\PEo}{\toq{P}_{\!\mathscrbf{E}^\bot}}
\newcommand{\PEzo}{\toq{P}_{\!\mathscrbf{E}_0^\bot}}
\newcommand{\PSz}{\toq{P}_{\!\mathscrbf{S}_0}}
\newcommand{\PSo}{\toq{P}_{\!\mathscrbf{S}^\bot}}

\title{Geometric variational principles for computational homogenization}

\author{C\'edric Bellis \& Pierre Suquet \\[4mm]
\small  Aix Marseille Univ, CNRS, Centrale Marseille, LMA, Marseille, France}

\begin{document}
\date{}
\maketitle

\begin{abstract}
The homogenization of periodic elastic composites is addressed through the reformulation of the local equations of the mechanical problem in a geometric functional setting. This relies on the definition of Hilbert spaces of kinematically and statically admissible tensor fields, whose orthogonality and duality properties are recalled. These are endowed with specific energetic scalar products that make use of a reference and uniform elasticity tensor. The corresponding strain and stress Green's operators are introduced and interpreted as orthogonal projection operators in the admissibility spaces. In this context and as an alternative to classical minimum energy principles, two geometric variational principles are investigated with the introduction of functionals that aim at measuring the discrepancy of arbitrary test fields to the kinematic, static or material admissibility conditions of the problem. By relaxing the corresponding local equations, this study aims in particular at laying the groundwork for the homogenization of composites whose constitutive properties are only partially known or uncertain. The local fields in the composite and their macroscopic responses are computed through the minimization of the proposed geometric functionals. To do so, their gradients are computed using the Green's operators and gradient-based optimization schemes are discussed. A FFT-based implementation of these schemes is proposed and they are assessed numerically on a canonical example for which analytical solutions are available.
\end{abstract}

\section{Introduction}

\subsection{Context}

Focusing on linear composite materials, the early works \cite{Kroner} and \cite{Willis77,Willis81} have shown that the local fields satisfying the governing equations of the associated mechanical problem are also solutions of some linear integral equations, which are reminiscent of the well-known Lippmann-Schwinger equation. These formulations rely typically on the introduction of a homogeneous comparison material and of the corresponding Green's operators for the strain or the stress fields. Remarkably, the Fourier transforms of the kernels of these integral operators are known in closed forms for different types of anisotropy of the reference medium, see \cite{Khatchaturyan,Mura}. These bases have enabled the development of methods aiming at computing local fields, and their macroscopic responses as well, as the solutions to these integral equations, starting from the work \cite{Moulinec94,Moulinec} where they are solved using a fixed-point iterative scheme in a FFT-based implementation. This method has developed in the field of computational homogenization with successful applications to a wide range of configurations and a concomitant improvement of the corresponding algorithms over the years, see \cite{Moulinec18} and the references therein. Some of these algorithms were developed without reference to variational principles  \cite{Moulinec94,MUL96,Eyre,Monchiet}, while others \cite{Michel,Brisard,Kabel} made explicit use of variational properties of the local fields.

A first link between such algorithms has been investigated in \cite{Brisard}. However, it is rather recently, see \cite{Kabel}, that the link between the Lippmann-Schwinger equation and the gradient of the strain-based energy functional has been evidenced. In particular, it has been shown that the iterative scheme introduced in \cite{Moulinec} can be interpreted as a gradient descent method with fixed step for this functional. The critical observation was that the gradient of this functional can be computed using the Green's operator when the space of second-order tensor fields is endowed with an energetic scalar product defined by the reference medium. A similar change of metric was previously used in \cite{Michel}. Doing so, it was then clear that the scheme of \cite{Moulinec} and its variations can be obtained directly in the form of gradient-based algorithms according to minimum energy principles. Moreover, the avenues for improvement from the original scheme with fixed step became clear, namely by using optimal or fast gradients methods as done subsequently in \cite{Schneider}. Conjugate-gradient methods have also been investigated in a number of earlier studies, see \cite{Zeman,Brisard,Gelebart}. However, the latter do not make use of the key property that the integral operator featured in the Lippmann-Schwinger equation is the gradient of the energy functional in a well-chosen Hilbert space and as such it is a self-adjoint operator. This fact has major implications for gradient descent methods, which results in subtle but fundamental differences between these algorithms.

The present study is structured around two key points: 
\begin{enumerate}
\item Within the framework of classical energetic variational principles the gradients of given functionals can be computed using the available Green's operators and efficient gradient-based minimization algorithms can be employed for fast and accurate computations of composites responses.
\item The computation of these gradients relies on endowing  the space of second-order tensor fields with a geometric, i.e., Hilbertian, structure. Such a functional framework is relatively well-known, see \cite{Milton}, and traces back to \cite{Moreau,SuquetTE,Suquet87}. What is crucial in the present study is the definition of well-chosen energetic scalar products so that the spaces of kinematically and statically tensors fields, together with their complementary subspaces, are linked by a number of orthogonality or duality properties, with the associated orthogonal projection operators being generated by the Green's operators.
\end{enumerate}
In this context, this study aims at blending these ideas together by formulating some geometrical variational principles that allow us to address the computational homogenization of composites from a new angle. Our objectives are detailed below, following a preliminary subsection to present the key elements of the problem.

\subsection{Preliminaries}

Consider a periodic composite material characterized by the unit-cell $\mathcal{V}\subset\mathbb{R}^d$ and the fourth-order elasticity tensor $\toq{L}(\tou{x})$ with major and minor symmetries. The local strain and stress fields $\tod{\eps}$ and $\tod{\sig}$ solve the so-called \emph{local problem} in $\mathcal{V}$ consisting of the compatibility equations, constitutive relations and equilibrium equations with periodic boundary conditions:
\begin{equation}
        \left\{\begin{aligned}
& \text{(i)} && \tod{\eps}(\tou{x}) = \macr{\tod{\eps}} + \tod{\eps}^*(\tou{x}), \quad  \tod{\eps}^*(\tou{x}) = \frac{1}{2} \big( \tod{\nabla} \tou{u}^*(\tou{x}) + \tod{\nabla} {\tou{u}^*}(\tou{x})^{\!\top}\big),\quad \tou{u}^* \ \text{periodic on}\ \partial \mathcal{V}, \\[1mm]
& \text{(ii)} &&  \tod{\sigma}(\tou{x})=\toq{L}(\tou{x})\tod{\eps}(\tou{x}), \\[2mm]
& \text{(iii)} && \operatorname{div}\tod{\sigma}(\tou{x})=\tou{0}, \quad \tod{\sig}\cdot\tou{n} \ \text{anti-periodic on}\ \partial \mathcal{V},
	\end{aligned}\right.
 \label{local1}
 \end{equation}
with $\tou{u}^*$ being the fluctuation of the displacement field in $\mathcal{V}$, $\tou{n}$ the unit outward normal on $\partial\mathcal{V}$ and $\macr{\tod{\eps}}$ an applied macroscopic strain. A mathematical definition of periodicity conditions is given in Appendix~\ref{math1}.

The average strain $\macr{\tod{\eps}}$ being prescribed then the \emph{effective} elasticity tensor $\eff{\toq{L}}$ is characterized by the energetic variational principle:
\begin{equation}
\frac{1}{2}\eff{\toq{L}}\,\macr{\tod{\eps}}:\macr{\tod{\eps}}=\min_{\substack{ \tod{e} \in \mathscrbf{E} \\ \langle\tod{e}\rangle=\macr{\tod{\eps}} }}\mathcal{J}(\tod{e}) \quad \text{with} \quad \mathcal{J}(\tod{e})=\frac{1}{2} \moy{\toq{L}(\tou{x})\tod{e}(\tou{x}):\tod{e}(\tou{x})},
\label{princip0}
\end{equation}
where $\mathscrbf{E}$ denotes the space of second-order tensor fields that are admissible strains, which will be properly defined in Section~\ref{sec:geo:set}, and with the averaging operator $\moy{\!\cdot\!}$ over $\mathcal{V}$ and the doubly contracted product given by
\[
 \moy{f}  = \frac{1}{|\mathcal{V}|} \int_\mathcal{V} f(\tou{x}) \dd\tou{x} \quad \text{and} \quad   \tod{s}\dc\tod{e}= \sum_{i,j=1}^d s_{ij}e_{ij} .
 \]
The actual strain field $\tod{\eps}$ solution of the local problem \eqref{local1} is the unique minimizer in \eqref{princip0}. Note that the tensor $\toq{L}(\tou{x})$ in (\ref{local1}.ii) is interpreted as a local operator hence the omission in such a relation of the doubly contracted product. We reserve the latter for products between second-order tensors.

\subsection{Objective}\label{sec:obj}

In this article, our objective is three-fold:
\begin{enumerate}
\item First, we aim at revisiting from a geometric standpoint the classical energetic variational approaches such as \eqref{princip0}, by relying on the geometric properties of the spaces of compatible strains and equilibrated stresses. In the minimum energy principle \eqref{princip0} the strain compatibility equation (\ref{local1}.i) is satisfied through the choice of the minimization space $\mathscrbf{E}$ and the constitutive relations (\ref{local1}.ii) are directly imposed in the definition of the cost functional $\mathcal{J}$. The equilibrium equation (\ref{local1}.iii) for the stress field, formally rewritten as the condition $\tod{\sigma}\in \mathscrbf{S}$ for the time being, with a proper definition of the space $\mathscrbf{S}$ given in the next section, is achieved through the minimization of $\mathcal{J}$ over $\mathscrbf{E}$. In this context, we show in Section \ref{sec:min:proj} that we can adopt an alternative approach using a variational principle of the form:
\begin{equation}\label{idea:var:1}
\tod{\eps} = \operatorname{arg}\underset{ \substack{ \tod{e} \in \mathscrbf{E} \\ \langle\tod{e}\rangle=\macr{\tod{\eps}} } }{\operatorname{min}}   \   \mathcal{N}(\tod{e}) \quad \text{with} \quad  \mathcal{N}(\tod{e})=\Updelta{\textrm{Equil}}(\toq{L} \tod{e} ) 
\end{equation}
where ``$\Updelta{\textrm{Equil}}$'' stands for a measure of the static admissibility of the test field defined as $\tod{s}=\toq{L} \tod{e}$. With $\mathscrbf{S}$ being the space of admissible stresses, achieving the condition $\tod{s}\in \mathscrbf{S}$ is equivalent to minimizing the norm of the projection of $\tod{s}$ onto the subspace orthogonal to $\mathscrbf{S}$. Upon introducing the corresponding orthogonal projection operator $\PSo$ and a suitable norm, then one defines $\mathcal{N}$ with 
\[
\Updelta{\textrm{Equil}}(\tod{s})=\frac{1}{2} \| \PSo \tod{s} \|^2,
\]
which justifies referring to \eqref{idea:var:1} as a geometric variational principle.\\

\item Building on this idea, we introduce next a two-field geometric variational principle that allows to treat the constitutive relations (\ref{local1}.ii) on the same level as the strain and the stress admissibility equations. This is of particular interest in the situations where the constitutive model is partially or fully unknown as when dealing with inverse problems of material identification, see \cite{Bonnet}, in data-driven computational approaches \cite{Kirchdoerfer}, or when material uncertainties must be accounted for, see \cite{Nouy,Staber} and the references therein. As shown in Section \ref{two:field:princip}, a two-field variational approach can be adopted through an \emph{unconstrained} minimization problem of the form
\begin{equation}\label{idea:var:2}
(\tod{\sig},\tod{\eps}) =  \operatorname{arg}\underset{\tod{s},\tod{e}}{\operatorname{min}}  \  \mathcal{P}(\tod{s},\tod{e})  \quad \text{with} \quad  \mathcal{P}(\tod{s},\tod{e})  = \Updelta{\textrm{Compat}}(\tod{e}) + \Updelta{\textrm{Const}}(\tod{s},\tod{e})  + \Updelta{\textrm{Equil}}(\tod{s}). 
\end{equation}
In \eqref{idea:var:2}, the functional $\mathcal{P}$ features the term $\Updelta{\textrm{Const}}(\tod{s},\tod{e})$ that is to be defined as a positive measure of the local error in constitutive relations between the test fields $\tod{s}$ and $\tod{e}$. The other two terms are geometric measures of the kinematic and static admissibilities of these fields: the third term is defined as in \eqref{idea:var:1} while the first one is a measure of the strain compatibility condition, i.e., $\tod{e}\in\mathscrbf{E}$ with $\moy{\tod{e}} = \macr{\tod{\eps}}$. By resorting to the projector $\PEo$ onto the subspace orthogonal to the space $\mathscrbf{E}$ of admissible strains and using a appropriate norm, then $\mathcal{P}$ is defined with
\[
\Updelta{\textrm{Compat}}(\tod{e})=\frac{1}{2} \| \moy{\tod{e}} - \macr{\tod{\eps}} + \PEo \tod{e} \|^2.
\]
\begin{figure}[htb]	
\centering
\includegraphics[width=0.65\textwidth]{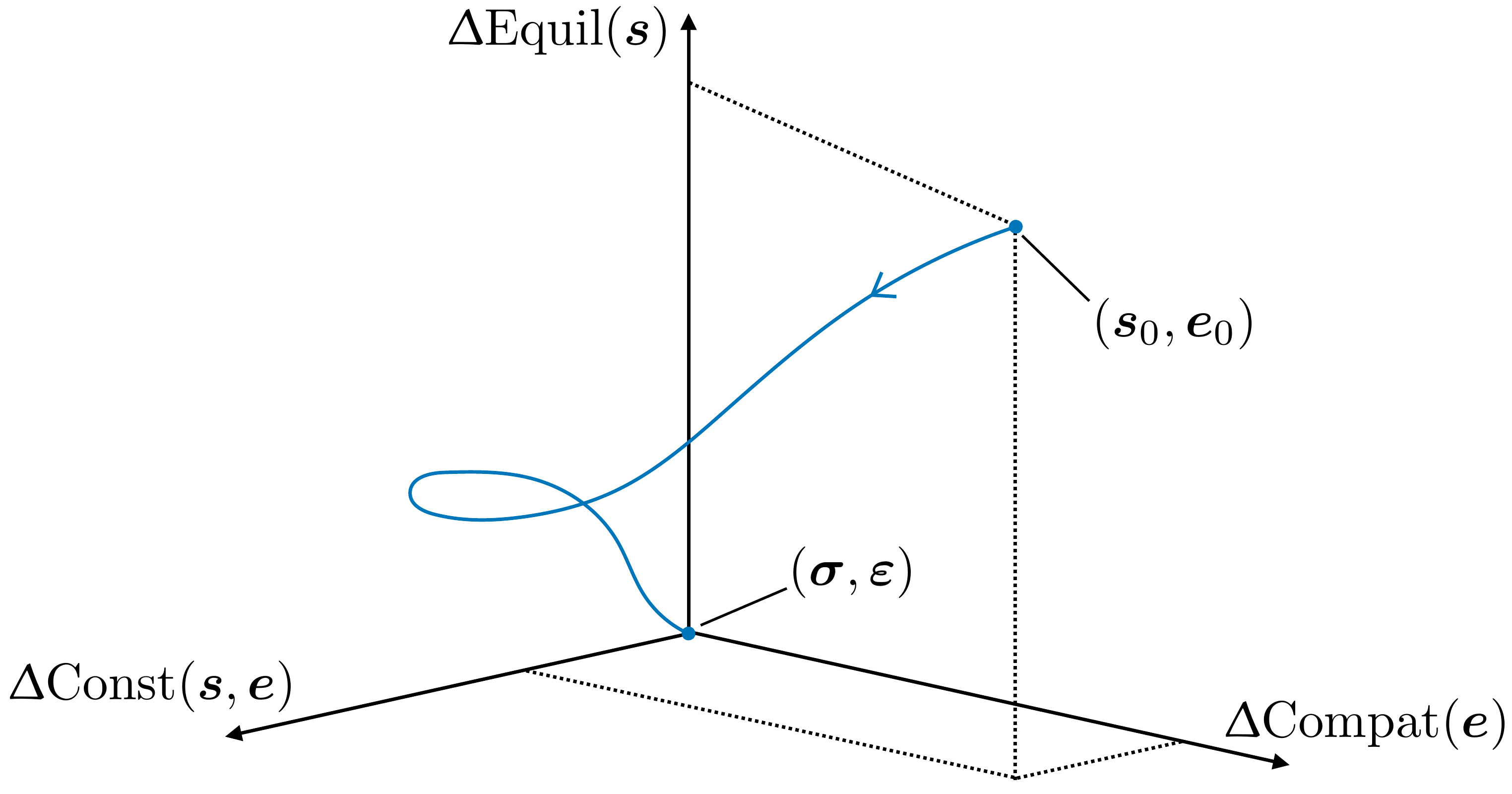}
\caption{Schematics of the evolution of computed fields $(\tod{s},\tod{e})$, from the initial point $(\tod{s}_0,\tod{e}_0)$ to the actual solution $(\tod{\sig},\tod{\eps})$, in a projection space indicating the discrepancy to the kinematic, static and material admissibility conditions.} \label{Fig:Schematics}
\end{figure}
According to the definition of the functional $\mathcal{P}$, any \emph{pair} of fields $(\tod{s},\tod{e})$ can be represented as a \emph{point} in a projection space where its coordinates correspond to the values of $\Updelta{\textrm{Compat}}(\tod{e})$, $\Updelta{\textrm{Const}}(\tod{s},\tod{e})$ and $\Updelta{\textrm{Equil}}(\tod{s})$, see the schematics of Fig. \ref{Fig:Schematics}. This allows in particular to visualize the evolution of the output of an iterative algorithm aiming at computing the solution to the minimization problem \eqref{idea:var:2} and to compare different computation strategies.

\item The solutions to the geometric variational principles \eqref{idea:var:1} and \eqref{idea:var:2} can be computed using gradient-based minimization algorithms. In this perspective, we show that the gradients of the functionals $\mathcal{N}$ and $\mathcal{P}$ can be expressed by means of the periodic strain and stress Green's operators, which will be defined in Section \ref{sec:Green:op}. In fact, those operators are shown to generate orthogonal projectors on the spaces of strain and stress tensor fields when the latter are endowed with well-chosen energetic scalar products. The obtained explicit forms of the gradients of the functionals $\mathcal{N}$ and $\mathcal{P}$ will also shed light on the proposed geometric variational principles by establishing some relationships with the classical minimum energy principles.
\end{enumerate}

The article is organized as follows. In Section \ref{sec:geo:Green} the geometric functional framework is set, the Green's operators are introduced and their properties as projection operators are investigated. Section \ref{sec:geo:var:princip} focuses on the introduction and the study of the geometric variational principles. Numerical implementation and examples are discussed in Section \ref{sec:num}. Mathematical definitions and classical properties of the Green's operators are deferred to the appendices.

\section{Strain and stress orthogonal decompositions}\label{sec:geo:Green}

\subsection{Geometric setting}\label{sec:geo:set}

Let $\toq{L}_0$ denote a uniform (no spatial dependence) elastic stiffness tensor, i.e., a positive definite fourth-order tensor with major and minor symmetries. One defines as $\mathscrbf{H}_e$ the space of symmetric $\mathscrbf{L}^2_\text{per}\!\left(\mathcal{V}\right)$-tensor fields, with the notation of Appendix \ref{math1}, when equipped with the following energetic scalar product
\begin{equation}
\pse{\tod{\eta}}{\tilde{\tod{\eta}}}= \moy{\toq{L}_0\tod{\eta}(\tou{x})\dc\tilde{\tod{\eta}}(\tou{x})}. \label{scalare}
\end{equation}
Doing so, $\mathscrbf{H}_e$ is a Hilbert space of \emph{strain} tensor fields. Let $\mathscrbf{U}_e\subset \mathscrbf{H}_e$ be the linear subspace of uniform strain fields (typically averages of local fields) and $\mathscrbf{E}_0\subset \mathscrbf{H}_e$ the linear subspace of kinematically compatible strain fields, which derive from a periodic displacement field:
\[
\mathscrbf{E}_0 = \left\{ \tod{e}^*\in \mathscrbf{H}_e  \text{ such that:}\  \exists\, \tou{w}^* \in \tou{H}^1_\text{per}\!\left(\mathcal{V}\right) ,\; \tod{e}^*(\tou{x}) = \frac{1}{2} \big(\tod{\nabla w}^*(\tou{x}) + \tod{\nabla} {\tou{w}^*}(\tou{x})^{\!\top}\big)  \text{ in}\ \mathcal{V} \right\}.
\]
By definition, every field $\tod{e}^*$ in $\mathscrbf{E}_0$ satisfies $\moy{\tod{e}^*(\tou{x})}=\tod{0}$. As $\mathscrbf{U}_e$ and $\mathscrbf{E}_0$ are closed subspaces of $\mathscrbf{H}_e$, they are Hilbert spaces for the scalar product \eqref{scalare}. Moreover, defining the subspace $ \mathscrbf{E}$ as 
 \begin{equation}
\mathscrbf{E} =   \mathscrbf{U}_e \oplus \mathscrbf{E}_0,
\end{equation}
then $\mathscrbf{H}_e$ admits the orthogonal decomposition
 \begin{equation}
\mathscrbf{H}_e = \mathscrbf{U}_e\oplus\mathscrbf{E}_0\oplus\mathscrbf{E}^\bot.
\label{decompHe}
\end{equation}
Therefore, any field $\tod{\eta} \in \mathscrbf{H}_e$ can be decomposed as
\begin{equation}
\tod{\eta} = \moy{\tod{\eta}} + \PEz\tod{\eta} + \PEo \tod{\eta}.
\label{prop6}
\end{equation}
with $\moy{\!\cdot\!}$, $\PEz$ and $\PEo$ being the orthogonal projection operators onto $\mathscrbf{U}_e$, $\mathscrbf{E}_0$ and $\mathscrbf{E}^\bot$ respectively, which are mutually orthogonal for the energetic scalar product \eqref{scalare}.

Its topological dual is denoted as $\mathscrbf{H}_s=\mathscrbf{H}'_e$. According to the Riesz representation theorem, there exists an isomorphic mapping $\toq{R}_s:\mathscrbf{H}_e\to\mathscrbf{H}_s$ such that for every $\tod{\eta}\in\mathscrbf{H}_e$:
\begin{equation}\label{Riesz}
\pse{\tod{\eta}}{\tilde{\tod{\eta}}}=\pdse{\toq{R}_s\tod{\eta}}{\tilde{\tod{\eta}}} \qquad \forall \tilde{\tod{\eta}}\in\mathscrbf{H}_e,
\end{equation}
where $\pdse{\cdot}{\cdot}$ denotes the duality product between $\mathscrbf{H}_s$ and $\mathscrbf{H}_e$. The definition of the energetic scalar product \eqref{scalare} allows to identify any element in $\mathscrbf{H}_s$ to a $\mathscrbf{L}^2_\text{per}\!\left(\mathcal{V}\right)$-tensor field that is dimensionally consistent with a \emph{stress} field. Moreover, the Riesz mapping reads $\toq{R}_s:\tod{\eta}\mapsto \toq{L}_0\tod{\eta}$ when the duality product is defined according to the principle of virtual work as
\[
\pdse{\tod{\tau}}{\tod{\eta}} = \moy{\tod{\tau}(\tou{x})\dc\tod{\eta}(\tou{x})}.
\]
Therefore, $\mathscrbf{H}_s$ is itself a Hilbert space with the induced energetic scalar product
\begin{equation}
\pss{\tod{\tau}}{\tilde{\tod{\tau}}}= \moy{\tod{\tau}(\tou{x})\dc\toq{L}_0^{-1}\tilde{\tod{\tau}}(\tou{x})}, \label{scalars}
\end{equation}
and, identifying $\mathscrbf{H}_e$ with its bidual, the inverse operator $\toq{R}_e=\toq{R}_s^{-1}:\mathscrbf{H}_s\to\mathscrbf{H}_e $ with $\toq{R}_e\tod{\tau}=\toq{L}_0^{-1}\tod{\tau}$ corresponds to the Riesz operator on $\mathscrbf{H}_s$.\\

As in \eqref{decompHe}, let $\mathscrbf{U}_s\subset \mathscrbf{H}_s$ be the linear subspace of uniform stress fields and $\mathscrbf{S}_0 $ the linear subspace of self-equilibrated fields, i.e., divergence-free and mean-free, so that
 \[
\mathscrbf{S}_0 = \Big\{ \tod{s} \in \mathscrbf{H}_s \text{ such that:}\  \operatorname{div}\tod{s}(\tou{x})=\tou{0} \text{ in}\ \mathcal{V}, \  \tod{s}\cdot\tou{n} \text{ anti-periodic on}\ \partial V, \  \moy{\tod{s}(\tou{x})}=\tod{0}\Big\}.
\]
The subspaces $\mathscrbf{U}_s$ and $\mathscrbf{S}_0$ are closed in $\mathscrbf{H}_s$ so that, defining the subspace $\mathscrbf{S}$ as
 \begin{equation}
\mathscrbf{S}=\mathscrbf{U}_s\oplus\mathscrbf{S}_0,
\end{equation}
one has the following orthogonal decomposition
 \begin{equation}
\mathscrbf{H}_s = \mathscrbf{U}_s\oplus\mathscrbf{S}_0\oplus\mathscrbf{S}^\bot
\label{decompHs}
\end{equation}
relatively to the energetic scalar product \eqref{scalars}. Thus, any field $\tod{\tau} \in \mathscrbf{H}_s$ can be decomposed as
\begin{equation}
\tod{\tau} = \moy{\tod{\tau}}  + \PSz \tod{\tau} +  \PSo \tod{\tau} ,
\label{prop7}
\end{equation}
with $\moy{\!\cdot\!}$, $\PSz$ and $\PSo$ being three mutually orthogonal projectors onto $\mathscrbf{U}_s$, $\mathscrbf{S}_0$ and $\mathscrbf{S}^\bot$. In this setting, the following lemma holds:
\begin{lemma}[Hill-type orthogonality properties] \label{Hill:orth}
\[\begin{aligned}
& \tod{e}\in \mathscrbf{E} && \Leftrightarrow && \pdse{\tod{s}}{\tod{e}}=0 \qquad \forall\tod{s}\in \mathscrbf{S}_0, \\
& \tod{s}\in \mathscrbf{S} && \Leftrightarrow && \pdse{\tod{s}}{\tod{e}}=0 \qquad \forall\tod{e}\in \mathscrbf{E}_0.
\end{aligned}\]
\end{lemma}
The proof of this lemma can be found in \cite{SuquetTE} and makes use of the characterization of distributions which are compatible strain fields in the sense of distributions \cite{Paris,Moreau}. It establishes that the space $\mathscrbf{E}$ (resp. $\mathscrbf{S}$) is the polar space of $\mathscrbf{S}_0$ (resp. $\mathscrbf{E}_0$). This implies, see \cite{Boffi}, the following duality characterizations of these closed spaces:
\begin{equation}
\mathscrbf{E}_0=(\mathscrbf{S}^\bot)' \quad \text{and} \quad \mathscrbf{S}_0=(\mathscrbf{E}^\bot)'.
 \label{dual:charac}
\end{equation}
Similar orthogonality properties that make use of the standard $L^2$-scalar product on the tensor space $\mathscrbf{L}^2_\text{per}\!\left(\mathcal{V}\right)$ can be found in particular in \cite{Suquet87} and \cite{Milton}. Lastly, Lemma \ref{Hill:orth} yields the original result of Hill transposed to periodic boundary conditions, i.e.,
\[
\forall \tod{s}\in \mathscrbf{S}, \;   \forall \tod{e} \in \mathscrbf{E}: \qquad \pdse{\tod{s}}{\tod{e}} = \moy{\tod{s}(\tou{x})\dc\tod{e}(\tou{x})} = \moy{\tod{s}(\tou{x})}\dc\moy{\tod{e}(\tou{x})}.
\]
With these definitions at hand, the local problem \eqref{local1} can be rewritten in the condensed form:
\begin{equation}
        \left\{\begin{aligned}
& \text{(i)} && \tod{\eps}\in \mathscrbf{E}, \quad \moy{\tod{\eps}}=\macr{\tod{\eps}}, \\[1mm]
& \text{(ii)} &&  \tod{\sigma}(\tou{x})=\toq{L}(\tou{x})\tod{\eps}(\tou{x}) \text{ in }\mathcal{V}, \\[2mm]
& \text{(iii)} && \tod{\sigma}\in \mathscrbf{S}.
	\end{aligned}\right.
\label{local:fct:space}
 \end{equation}
 \begin{remark}
From now on, we reserve the notation $\tod{\eps},\, \tod{\sig}$ for the actual solution to \eqref{local:fct:space}. Tensor fields that belong to the subspaces $\mathscrbf{E}$ and $\mathscrbf{S}$ are referred to as \emph{admissible} strain and stress, and they are denoted as $\tod{e},\, \tod{s}$ respectively. The notations $\tod{\eta},\,\tod{\tau}$ are used for arbitrary fields in $\mathscrbf{H}_e$ and $\mathscrbf{H}_s$.
\end{remark}
 
\subsection{Green's operators and orthogonal projectors}\label{sec:Green:op}

Two (periodic) Green's operators, $\toq{\Gamma}_0$ for the strain field and $\toq{\Delta}_0$ for the stress field, can be associated with $\toq{L}_0$ and $\toq{L}_0^{-1}$. More specifically, for a given field $\tod{\tau}$ in $\mathscrbf{H}_s$, consider the following Eshelby problem:
\begin{equation}
\text{Find }  \tod{e}^*\in  \mathscrbf{E}_0 \text{ such that }  \tod{s}\overset{\operatorname{def}}{=}(\toq{L}_0 \tod{e}^*-\tod{\tau})   \in \mathscrbf{S}.
 \label{thermoelas}
\end{equation}
The problem (\ref{thermoelas}) has a unique solution $\tod{e}^* \in \mathscrbf{E}_0$. This allows to define the Green's operators as follows:
\begin{definition}\label{def:Green}
The periodic strain Green's operator $\toq{\Gamma}_0:\mathscrbf{H}_s\to\mathscrbf{H}_e$ of the reference medium with stiffness $\toq{L}_0$ is defined as 
\begin{equation}
\toq{\Gamma}_0:  \tod{\tau} \mapsto\toq{\Gamma}_0\tod{\tau} = \tod{e}^* \text{such that }\tod{e}^* \text{ is the solution of (\ref{thermoelas})}.
 \label{Gamma0}
\end{equation}
The periodic stress Green's operator $\toq{\Delta}_0:\mathscrbf{H}_e\to\mathscrbf{H}_s$ is defined in a similar way as 
\begin{equation}
\toq{\Delta}_0:  \tod{\eta}  \mapsto \toq{\Delta}_0 \tod{\eta}=\tod{s}^* \text{ such that } \tod{s}^* \in \mathscrbf{S}_0  \text{ and } \tod{e}\overset{\operatorname{def}}{=}(\toq{L}_0^{-1}\tod{s}^*-\tod{\eta}) \in \mathscrbf{E}. 
\label{Delta0}
\end{equation}
\end{definition}

Classical properties of $\toq{\Gamma}_0$ (and similarly of $\toq{\Delta}_0$) are summarized in Appendix \ref{app:Green}. As discussed in the previous section, the strain and the stress have different dimensions and live in the dual Hilbert spaces $\mathscrbf{H}_e$ and $\mathscrbf{H}_s$ respectively, each being endowed with its own energetic scalar product \eqref{scalare} or \eqref{scalars}. It will now be seen that these energetic scalar products generate two Helmholtz decompositions that are associated with \eqref{decompHe} and \eqref{decompHs}. Such decompositions will be expressed in terms of the operators $\toq{\Gamma}_0 \toq{L}_0$ and $\toq{\Delta}_0 \toq{L}_0^{-1}$ so that our aim is now to study some useful properties of the latter.

\begin{lemma}\label{lemma:G0L0}
Considering the Hilbert space $\mathscrbf{H}_e$ endowed with the energetic scalar product \eqref{scalare}, then the operator $\toq{\Gamma}_0 \toq{L}_0 : \mathscrbf{H}_e\to\mathscrbf{H}_e$ is the orthogonal \emph{strain} projector onto $\mathscrbf{E}_0$ as
\begin{enumerate}
\item $\toq{\Gamma}_0 \toq{L}_0$ is idempotent, i.e., $\toq{\Gamma}_0 \toq{L}_0\toq{\Gamma}_0 \toq{L}_0=\toq{\Gamma}_0 \toq{L}_0$.
\item $\toq{\Gamma}_0 \toq{L}_0$ is self-adjoint, i.e.,
\begin{equation}
\pse{\tod{\eta}}{\toq{\Gamma}_0 \toq{L}_0\tilde{\tod{\eta}}} = \pse{ \toq{\Gamma}_0 \toq{L}_0\tod{\eta}}{\tilde{\tod{\eta}}}, \qquad \forall \tod{\eta}, \tilde{\tod{\eta}} \in \mathscrbf{H}_e.
\label{self20}
\end{equation}
\item For all $\tod{\eta}  \in \mathscrbf{H}_e$ then $\toq{\Gamma}_0 \toq{L}_0\tod{\eta}\in \mathscrbf{E}_0$ and for all $\tod{e}^* \in \mathscrbf{E}_0$ it holds $\toq{\Gamma}_0 \toq{L}_0 \tod{e}^* = \tod{e}^*$.
\end{enumerate}
\end{lemma}
\begin{proof}  1. The idempotence of $\toq{\Gamma}_0 \toq{L}_0$ is a direct consequence of (\ref{prop3bis}) in Appendix \ref{app:Green}.\\
2. A straightforward calculation shows that 
\[
\pse{\tod{\eta}}{ \toq{\Gamma}_0 \toq{L}_0 \tilde{\tod{\eta}}} = \pdse{\toq{L}_0\tod{\eta}} {\toq{\Gamma}_0\toq{L}_0\tilde{\tod{\eta}}} = \pdse{\tod{\tau}}{ \toq{\Gamma}_0 \tilde{\tod{\tau}}},
\]
where $\tod{\tau} = \toq{L}_0\tod{\eta}$ and $\tilde{\tod{\tau}} = \toq{L}_0\tilde{\tod{\eta}}$. Then \eqref{self20} follows from the reciprocity identity satisfied by $\toq{\Gamma}_0$, see Lemma \ref{lemma:Green} in Appendix \ref{app:Green}.\\
3. The third item is a direct consequence of the definition of $\toq{\Gamma}_0$ and of Property 2 in \cite{Michel}. This allows to conclude, see \cite{Brezis}, that $\toq{\Gamma}_0 \toq{L}_0$ is a projection operator from $\mathscrbf{H}_e$ onto the subspace $\mathscrbf{E}_0$, which is orthogonal for the energetic scalar product \eqref{scalare}.
\end{proof}

A similar lemma can be proved for the operator $\toq{\Delta}_0 \toq{L}_0^{-1}:\mathscrbf{H}_s\to\mathscrbf{H}_s$ using the duality principle of \cite{Milton} and provided that $\mathscrbf{H}_s$ is endowed with the energetic scalar product \eqref{scalars}. In particular, $\toq{\Delta}_0 \toq{L}_0^{-1}$ is the orthogonal \emph{stress} projector from $\mathscrbf{H}_s$ onto $\mathscrbf{S}_0$ for the energetic scalar product \eqref{scalars}. As a consequence of Lemma \ref{lemma:G0L0}, one arrives at the main result of this section:

\begin{proposition}\label{gen:decomp}
Considering the orthogonal decomposition \eqref{prop6} such that $\tod{\eta} = \moy{\tod{\eta}} + \PEz\tod{\eta} + \PEo \tod{\eta}$ for all $\tod{\eta} \in \mathscrbf{H}_e$, then the featured orthogonal projection operators can be expressed in terms of the Green's operators as
\begin{equation}
\PEz=\toq{\Gamma}_0 \toq{L}_0 \qquad \text{and} \qquad \PEo = \toq{L}_0^{-1} \toq{\Delta}_0.
\label{proj1}
\end{equation}
Similarly, the orthogonal projectors that enter the decomposition \eqref{prop7}, i.e., $\tod{\tau} = \moy{\tod{\tau}}  + \PSz \tod{\tau} +  \PSo \tod{\tau}$ for all $\tod{\tau} \in \mathscrbf{H}_s$, are given by
\begin{equation}
\PSz = \toq{\Delta}_0 \toq{L}_0^{-1} \qquad \text{and} \qquad \PSo=\toq{L}_0 \toq{\Gamma}_0.
\label{proj2}
\end{equation}
\end{proposition}

\begin{proof}
For all $\tod{\eta} \in \mathscrbf{H}_e$, Definition \ref{def:Green} entails that there exists $\tod{e}\in\mathscrbf{E}$ such that $\toq{L}_0\tod{e} = \toq{\Delta}_0\tod{\eta}-\toq{L}_0\tod{\eta}$.
Using \eqref{Gamma0} with $\tod{\tau}=\toq{L}_0\tod{e}$ then, on the one hand, there exists $\tod{s}\in\mathscrbf{S}$ such that
\begin{equation}\label{id:int}
\tod{s}=\toq{L}_0\toq{\Gamma}_0\tod{\tau}-\tod{\tau}=\toq{L}_0\toq{\Gamma}_0(\toq{\Delta}_0\tod{\eta}-\toq{L}_0\tod{\eta})-(\toq{\Delta}_0\tod{\eta}-\toq{L}_0\tod{\eta})=-\toq{L}_0\toq{\Gamma}_0\,\toq{L}_0\tod{\eta}-\toq{\Delta}_0\tod{\eta}+\toq{L}_0\tod{\eta},
\end{equation}
using the relation $ \toq{\Gamma}_0 \toq{\Delta}_0= \tod{0}$, consequence of Lemma \ref{lemma:Green} in Appendix \ref{app:Green}. On the other hand, introducing the decomposition $\tod{e}=\moy{\tod{e}}+\tod{e}^*$ with $\moy{\tod{e}}\in\mathscrbf{U}_e$ and $\tod{e}^*\in\mathscrbf{E}_0$, then one has 
\[
\tilde{\tod{e}}\overset{\operatorname{def}}{=}\toq{\Gamma}_0\tod{\tau}=\toq{\Gamma}_0\toq{L}_0\tod{e}=\toq{\Gamma}_0\toq{L}_0\big(\moy{\tod{e}}+\tod{e}^*\big)=\tod{e}^*.
\]
Therefore, one obtains that
\[
\tod{s}=\toq{L}_0\tilde{\tod{e}}-\tod{\tau}=\toq{L}_0(\tod{e}^*-\tod{e})=-\toq{L}_0\moy{\tod{e}}.
\]
Finally, according to \eqref{Delta0} one has $\langle\tod{e}\rangle=-\langle\tod{\eta}\rangle$ so that $\tod{s}=\toq{L}_0\langle\tod{\eta}\rangle$. Combining this last equation with \eqref{id:int} yields
\[
\toq{L}_0\langle\tod{\eta}\rangle=-\toq{L}_0\toq{\Gamma}_0\,\toq{L}_0\tod{\eta}-\toq{\Delta}_0\tod{\eta}+\toq{L}_0\tod{\eta},
\]
which after multiplication by $\toq{L}_0^{-1}$ yields the sought identity \eqref{prop6} by defining the orthogonal projections as $\moy{\!\cdot\!}$, $\PEz=\toq{\Gamma}_0 \toq{L}_0$ and $\PEo=\toq{L}_0^{-1} \toq{\Delta}_0$. The fact that they are mutually orthogonal follows from the definitions of the Green's operators as well as \eqref{prop3}. The identity \eqref{prop7} with \eqref{proj2} is obtained by duality.
\end{proof}

\begin{remark}
The strain and stress Hilbert spaces $\mathscrbf{H}_e$ and $\mathscrbf{H}_s$ differ by the physical dimension of their elements. Noticeably, one could avoid working with the two different spaces by considering the space of $\mathscrbf{L}^2_\text{per}\!\left(\mathcal{V}\right)$-tensor fields endowed with the standard $L^2$-scalar product. Based on Proposition \ref{gen:decomp}, when \eqref{prop6} is multiplied by $\toq{L}_0^{1/2}$ or \eqref{prop7} by $\toq{L}_0^{-1/2}$ then a single orthogonal decomposition is obtained as
\begin{equation}
\forall \tod{\zeta} \in \mathscrbf{L}^2_\text{per}\!\left(\mathcal{V}\right): \quad \tod{\zeta} = \moy{\tod{\zeta}} + \toq{L}_0^{1/2}\toq{\Gamma}_0 \toq{L}_0^{1/2}\tod{\zeta} + \toq{L}_0^{-1/2} \toq{\Delta}_0\toq{L}_0^{-1/2} \tod{\zeta}.
\label{gen:decomp:L2}
\end{equation}
In \eqref{gen:decomp:L2}, the operators $\moy{\!\cdot\!}$, $\toq{L}_0^{1/2}\toq{\Gamma}_0 \toq{L}_0^{1/2}$ and $\toq{L}_0^{-1/2} \toq{\Delta}_0\toq{L}_0^{-1/2}$ are three mutually orthogonal projectors in $\mathscrbf{L}^2_\text{per}\!\left(\mathcal{V}\right)$. However, according to this decomposition, a field $\tod{\zeta} \in \mathscrbf{L}^2_\text{per}\!\left(\mathcal{V}\right)$ would have neither the dimension of a strain nor of a stress but be consistent with either
\[
\tod{\zeta}=\toq{L}_0^{1/2}\tod{\eta} \qquad \text{or} \qquad \tod{\zeta}=\toq{L}_0^{-1/2}\tod{\tau}
\]
for $\tod{\eta}\in\mathscrbf{H}_e$ or $\tod{\tau}\in\mathscrbf{H}_s$. For this particular reason, we prefer to keep working with the spaces $\mathscrbf{H}_e$ and $\mathscrbf{H}_s$ that are dimensionally consistent with the mechanical problem considered.
\end{remark}

\section{Geometric variational principles}\label{sec:geo:var:princip}

\subsection{Strain-based variational principles}\label{sec:min:proj}

\subsubsection{Minimum energy principle}

With a slight abuse of notation, the energetic functional in \eqref{princip0} is redefined as $\mathcal{J}:\mathscrbf{E}_0\to\mathbb{R}$ with
\begin{equation}
\mathcal{J}(\tod{e}^*)=\frac{1}{2} \pdse{\toq{L}(\macr{\tod{\eps}} + \tod{e}^*)}{\macr{\tod{\eps}} + \tod{e}^*},
\label{cost:fct:J}
\end{equation}
with $\macr{\tod{\eps}}\in\mathscrbf{U}_e$ given. Doing so, the energetic variational principle \eqref{princip0} is equivalent to the minimization of the functional \eqref{cost:fct:J} with respect to $\tod{e}^*\in \mathscrbf{E}_0$. As discussed in the introduction, this optimization constraint together with the definition of $\mathcal{J}$ itself enforce both conditions (\ref{local:fct:space}.i) and (\ref{local:fct:space}.ii). The remaining equilibrium equation (\ref{local:fct:space}.iii) for the stress field, i.e., $\tod{\sigma}\in \mathscrbf{S}$, is achieved through the minimization of $\mathcal{J}$ over $\mathscrbf{E}_0$ as the space $\mathscrbf{S}$ constitutes the polar space of the former. In this context, the gradient of $\mathcal{J}$ at $\tod{e}^*\in\mathscrbf{E}_0$ can be computed. It is defined as the element of $\mathscrbf{E}_0$ such that
\[
\pse{ \tou{\nabla}\mathcal{J} (\tod{e}^*)}{\tilde{\tod{e}}^*} = \lim_{t\rightarrow 0} \frac{\mathcal{J}(\tod{e}^* + t \tilde{\tod{e}}^* )-\mathcal{J}(\tod{e}^*)}{t} = 
\pdse{\toq{L}\tod{e}}{\tilde{\tod{e}}^*}=\pse{ \toq{L}_0^{-1}\toq{L}\tod{e}  }{ \tilde{\tod{e}}^* } \qquad \forall \tilde{\tod{e}}^* \in \mathscrbf{E}_0,
\]
where $\tod{e}=\macr{\tod{\eps}} + \tod{e}^*$. This entails that $\big(\tou{\nabla}\mathcal{J} (\tod{e}^*)-\toq{L}_0^{-1}\toq{L}\tod{e} \big)\in\mathscrbf{E}_0^\bot$, or equivalently
\[
\PEz\big(\tou{\nabla}\mathcal{J} (\tod{e}^*)-\toq{L}_0^{-1}\toq{L}\tod{e} \big)=\tod{0}.
\]
As $\tou{\nabla}\mathcal{J} (\tod{e}^*)\in\mathscrbf{E}_0$ and $\PEz\toq{L}_0^{-1}\toq{L}\tod{e}=\toq{\Gamma}_0\toq{L}\tod{e}$, it implies that
\begin{equation}\label{grad:J}
\tou{\nabla}\mathcal{J} (\tod{e}^*) = \toq{\Gamma}_0 \toq{L} (\macr{\tod{\eps}} + \tod{e}^*).
\end{equation}
Note that the identity \eqref{grad:J} was first reported in \cite{Kabel} where the scalar product \eqref{scalare} is recovered through the change of a metric.

The necessary optimality conditions corresponding to the minimization of the functional $\mathcal{J}$ reads
\begin{equation}
\tou{\nabla}\mathcal{J} (\tod{\eps}^*) = \tod{0}, \qquad \text{i.e.,} \qquad \toq{\Gamma}_0 \toq{L} (\macr{\tod{\eps}} + \tod{\eps}^*)=\tod{0}.
 \label{optimality0}
\end{equation}
Using the properties of the Green's operator $\toq{\Gamma}_0$ such that for all $\tod{\eta}\in\mathscrbf{H}_e$ one has
\[
\toq{\Gamma}_0\toq{L}\tod{\eta}=\toq{\Gamma}_0\toq{\delta L}\tod{\eta}+\PEz\tod{\eta},
\]
with $\toq{\delta L}(\tou{x})=\toq{L}(\tou{x}) -\toq{L}_0$, and according to the decomposition $\tod{\eps}=\macr{\tod{\eps}} + \tod{\eps}^*$, then \eqref{optimality0} is equivalent to
\begin{equation}
(\toq{I} + \toq{\Gamma}_0 \toq{\delta L})\tod{\eps}  = \macr{\tod{\eps}}
 \label{LS1}
\end{equation}
where $\toq{I}$ is the identity operator. This equation coincides with the (periodic) Lippmann-Schwinger equation used by \cite{Moulinec} in linear composites.

\begin{remark}
Consider a composite made of non-linear constituents with a local energy density $w(\tod{x},\tod{e})$ which is convex with respect to $\tod{e}$ and bounded from above and below by two quadratic functions $\frac{1}{2}L\tod{e}:\tod{e}$ and $\frac{1}{2}\mu\tod{e}:\tod{e}$ respectively, such as $\big(\frac{1}{2}L\tod{e}:\tod{e}-w(\tod{x},\tod{e})\big)$ and $\big(w(\tod{x},\tod{e})-\frac{1}{2}\mu\tod{e}:\tod{e}\big)$ are convex. Then the response of the composite derives from an effective potential $\eff{w}$ defined as
\[
\eff{w}(\macr{\tod{\eps}})=\min_{ \tod{e}^*\in \mathscrbf{E}_0 }\mathcal{J}(\tod{e}^*) \quad \text{with} \quad \mathcal{J}(\tod{e}^*) = \moy{ w(\tod{x},\macr{\tod{\eps}}+\tod{e}^*) }.
\]
In this context, it can be proved that
\[
\tou{\nabla}\mathcal{J} (\tod{e}^*) = \toq{\Gamma}_0 \, \partial_{\tod{e}}w(\tod{x},\macr{\tod{\eps}}+\tod{e}^*)
\]
in $\mathscrbf{E}_0$ endowed with the energetic scalar product \eqref{scalare}.
\end{remark}

\subsubsection{Geometric alternative to the energetic principle}\label{sec:connect:en}

Rather than minimizing the energetic functional $\mathcal{J}$ in \eqref{cost:fct:J}, one can adopt a geometric approach by minimizing the norm of its gradient \eqref{grad:J}. To do so, consider the geometric functional $\mathcal{N}:\mathscrbf{E}_0\to\mathbb{R}$ defined as
\begin{equation}
\mathcal{N}(\tod{e}^*) = \frac{1}{2} \| \tou{\nabla}\mathcal{J} (\tod{e}^*) \|^2_{\mathscrbf{H}_e} = \frac{1}{2} \| \toq{\Gamma}_0\toq{L}(\macr{\tod{\eps}} + \tod{e}^*) \|^2_{\mathscrbf{H}_e} \qquad \forall \tod{e}^*\in\mathscrbf{E}_0,
\label{cost:fct:N:bis}
\end{equation}
where $\macr{\tod{\eps}}\in\mathscrbf{U}_e$ is given. According to the Riesz mapping \eqref{Riesz}, it is seen that $\|\tod{\eta} \|_{\mathscrbf{H}_e}=\|\toq{L}_0\tod{\eta} \|_{\mathscrbf{H}_s}$ for all $\tod{\eta}\in\mathscrbf{H}_e$. Moreover, Proposition \ref{gen:decomp} shows that the orthogonal projector from $\mathscrbf{H}_s$ onto $\mathscrbf{S}^\bot$ for the scalar product \eqref{scalars} is given by $\PSo=\toq{L}_0 \toq{\Gamma}_0$. Therefore, \eqref{cost:fct:N:bis} can be recast as
\begin{equation}
\mathcal{N}(\tod{e}^*)=\frac{1}{2} \| \PSo\toq{L}(\macr{\tod{\eps}} + \tod{e}^*) \|^2_{\mathscrbf{H}_s},
\label{cost:fct:N}
\end{equation}
with the strain field solution $\tod{\eps}^*$ being characterized by the variational problem 
\begin{equation}
\tod{\eps}^* = \operatorname{arg}\underset{ \tod{e}^* \in \mathscrbf{E}_0 }{\operatorname{min}}   \   \mathcal{N}(\tod{e}^*).
\label{princip:geo1}
\end{equation}
The stationary value of $\mathcal{N}$ is zero, which is reached when $\toq{L}\tod{e}\in\mathscrbf{S}$ with $\tod{e}=\macr{\tod{\eps}} + \tod{e}^*$. This constitutes the sought geometric variational principle where the stress admissibility condition is achieved through the minimization of the norm of the projection of the test field defined as $\tod{\tau}=\toq{L}(\macr{\tod{\eps}} + \tod{e}^*)$ in the orthogonal space $\mathscrbf{S}^\bot$. This justifies the notation used in Section \ref{sec:obj}, i.e., $\mathcal{N}(\tod{e}^*)=\Updelta{\textrm{Equil}}(\toq{L}\big(\macr{\tod{\eps}} + \tod{e}^*) \big)$ with ``$\Updelta{\textrm{Equil}}$'' being a measure of the static admissibility of $\tod{\tau}=\toq{L}(\macr{\tod{\eps}} + \tod{e}^*)$.\\

Lastly, note that for the strain-based geometric variational principle (\ref{cost:fct:N}--\ref{princip:geo1}), the effective properties $\eff{\toq{L}}$ are computed according to the energetic identity \eqref{princip0}. In this context, the main result of this section is the following:

\begin{proposition}\label{grad:N}
The gradient of $\mathcal{N}$ in $\mathscrbf{E}_0$ endowed with the energetic scalar product \eqref{scalare}, and at $\tod{e}^*$, is the element of $\mathscrbf{E}_0$ defined as
\[
\tou{\nabla}\mathcal{N}(\tod{e}^*)=\toq{\Gamma}_0\toq{L}\toq{\Gamma}_0\toq{L}(\macr{\tod{\eps}} + \tod{e}^*).
\]
where $\toq{\Gamma}_0$ is the strain Green's operator associated with the reference elastic medium with stiffness $\toq{L}_0$.   
\end{proposition}
\begin{proof}
 The gradient of $\mathcal{N}$ in $\mathscrbf{E}_0$ endowed with the energetic scalar product \eqref{scalare} is defined as the element of $\mathscrbf{E}_0$ such that
 \[
\pse{ \tou{\nabla}\mathcal{N} (\tod{e}^*)}{\tilde{\tod{e}}^*} = \lim_{t\rightarrow 0} \frac{\mathcal{N}(\tod{e}^* + t \tilde{\tod{e}}^* )-\mathcal{N}(\tod{e}^*)}{t} = 
\pse{\toq{\Gamma}_0\toq{L}\tod{e}}{\toq{\Gamma}_0\toq{L}\tilde{\tod{e}}^*} \qquad \forall \tilde{\tod{e}}^* \in \mathscrbf{E}_0,
\]
where $\tod{e}=\macr{\tod{\eps}} + \tod{e}^*$. Making use of the properties of $\toq{\Gamma}_0$, then straightforward calculations lead to 
\[
\pse{\tou{\nabla}\mathcal{N}(\tod{e}^*)}{\tilde{\tod{e}}^*} = \pdse{\toq{L}\toq{\Gamma}_0\toq{L}\tod{e} }{\tilde{\tod{e}}^*}=\pse{\toq{L}_0^{-1}\toq{L}\toq{\Gamma}_0\toq{L}\tod{e}}{\tilde{\tod{e}}^*} \qquad \forall \tilde{\tod{e}}^* \in \mathscrbf{E}_0.
\]
The above equation implies that $\big(\tou{\nabla}\mathcal{N} (\tod{e}^*)-\toq{L}_0^{-1}\toq{L}\toq{\Gamma}_0\toq{L}\tod{e} \big)\in\mathscrbf{E}_0^\bot$ and therefore
\[
\PEz\big(\tou{\nabla}\mathcal{N} (\tod{e}^*)-\toq{L}_0^{-1}\toq{L}\toq{\Gamma}_0\toq{L}\tod{e} \big)=\tod{0}.
\]
As $\tou{\nabla}\mathcal{N} (\tod{e}^*)\in\mathscrbf{E}_0$ and $\PEz\toq{L}_0^{-1}\toq{L}\toq{\Gamma}_0\toq{L}\tod{e} =\toq{\Gamma}_0\toq{L}\toq{\Gamma}_0\toq{L}\tod{e}$, this leads to
\[
\tou{\nabla}\mathcal{N}(\tod{e}^*)-\toq{\Gamma}_0\toq{L}\toq{\Gamma}_0\toq{L}\tod{e}=\tod{0}.
\]
\end{proof}

At this point one has shown that the strain field $\tod{\eps}=\macr{\tod{\eps}}+\tod{\eps}^*$ solution of the local problem \eqref{local1} with imposed macroscopic strain $\macr{\tod{\eps}}\in\mathscrbf{U}_e$ can be characterized by either the energetic or the geometric variational principles, using the functionals \eqref{cost:fct:J} or \eqref{cost:fct:N} respectively, which pertain both to the fluctuating term $\tod{e}^*\in\mathscrbf{E}_0$. To shed light on the connection between these two variational principles we consider that $\tod{\eps}^*$ will be computed by an iterative gradient-based minimization scheme. The following proposition establishes a link between the two variational principles.
\begin{proposition}\label{prop:equiv:J:N}
For any $\tod{e}^*\in\mathscrbf{E}_0$, the vector $-\tou{\nabla}\mathcal{N}(\tod{e}^*)$ is a descent direction for the energetic functional $\mathcal{J}$ and, reciprocally, $-\tou{\nabla}\mathcal{J}(\tod{e}^*)$ is a descent direction for the geometric functional $\mathcal{N}$.
\end{proposition}

\begin{proof}
A descent direction for a given functional, say $\mathcal{J}$, at $\tod{e}^*$ is defined as a tensor $\tod{p}\in\mathscrbf{E}_0$ such that $\pse{\tod{p}}{\tou{\nabla}\mathcal{J}(\tod{e}^*)}<0$. Owing to Proposition \ref{grad:N} and Equation \eqref{grad:J}, then for any $\tod{e}^*\in\mathscrbf{E}_0$, one has
\[
-\pse{\tou{\nabla}\mathcal{J}(\tod{e}^*)}{\tou{\nabla}\mathcal{N}(\tod{e}^*)}=-\pse{\toq{\Gamma}_0\toq{L}\tod{e}}{\toq{\Gamma}_0\toq{L}\toq{\Gamma}_0\toq{L}\tod{e}} = -\pse{\tilde{\tod{e}}^*}{\toq{\Gamma}_0\toq{L}\tilde{\tod{e}}^*}
\]
with $\tod{e}=\macr{\tod{\eps}}+\tod{e}^*$ and where we introduced $\tilde{\tod{e}}^*=\toq{\Gamma}_0\toq{L}\tod{e} \in \mathscrbf{E}_0$. The properties of $\toq{\Gamma}_0$ entail
\[
\pse{\tilde{\tod{e}}^*}{\toq{\Gamma}_0\toq{L}\tilde{\tod{e}}^*}=\pdse{\toq{L}\tilde{\tod{e}}^*}{\toq{\Gamma}_0\toq{L}_0\tilde{\tod{e}}^*}=\pdse{\toq{L}\tilde{\tod{e}}^*}{\tilde{\tod{e}}^*}.
\]
The positive definiteness of the quadratic form $\pdse{\toq{L}\tilde{\tod{e}}^*}{\tilde{\tod{e}}^*}$ allows to conclude that
\[
-\pse{\tou{\nabla}\mathcal{J}(\tod{e}^*)}{\tou{\nabla}\mathcal{N}(\tod{e}^*)} <0.
\]
\end{proof}
This result shows that a gradient descent algorithm based on either $\tou{\nabla}\mathcal{J}$ or $\tou{\nabla}\mathcal{N}$ would result in the simultaneous minimization of both functionals $\mathcal{J}$ and $\mathcal{N}$.

\subsection{A two-field variational principle}\label{two:field:princip}

As discussed in Section \ref{sec:obj}, there are situations where the constitutive relations are partially, or even fully, unknown. In this context, we now introduce a variational principle that allows leeway in treating these relations through a proper minimization, as it is done for the stress and strain admissibility conditions.

\subsubsection{Minimum projections principle}\label{sec:min:proj:two:field}

First, the case of linear constituents is considered and the local problem \eqref{local1} with prescribed macroscopic strain $\macr{\tod{\eps}}\in\mathscrbf{U}_e$ is addressed in its condensed form \eqref{local:fct:space}. On the one hand, to deal with Eqn. (\ref{local:fct:space}.ii) in a variational setting we draw from the concept of error in constitutive relations, initially introduced in \cite{LL83} for error estimation in the finite element method, and define the following functional locally for all $\tou{x}\in\mathcal{V}$:
\begin{equation}\label{ecr:lin:0}
r(\tou{x},\tod{\tau},\tod{\eta})=\frac{1}{2}\toq{L}\tod{\eta}:\tod{\eta} + \frac{1}{2}\tod{\tau}:\toq{L}^{-1}\tod{\tau} - \tod{\tau}:\tod{\eta} \qquad \forall (\tod{\tau},\tod{\eta})\in\mathscrbf{H}_s\times\mathscrbf{H}_e.
\end{equation}
While expressed in this form this functional can be easily generalized to non-linear composites, see the next section, it can be conveniently rewritten in the linear case as
\begin{equation}\label{ecr:lin}
r(\tou{x},\tod{\tau},\tod{\eta})= \frac{1}{2} \big(\tod{\tau} - \toq{L}\tod{\eta} \big):  \toq{L}^{-1}\big( \tod{\tau} - \toq{L}\tod{\eta}\big),
\end{equation}
which makes clear that $r(\tou{x},\tod{\tau},\tod{\eta})\geq0$ in $\mathcal{V}$, while $r(\tou{x},\tod{\tau},\tod{\eta})=0$ locally if and only if $\tod{\tau}(\tou{x}) = \toq{L}(\tou{x})\tod{\eta}(\tou{x})$. Note that we use the tensor $\toq{L}^{-1}$ to define the quadratic form \eqref{ecr:lin}, rather than the reference tensor $\toq{L}_0^{-1}$ that is used in the energetic scalar product \eqref{scalars}. Our motivation to do so is to be able to generalize the formulation \eqref{ecr:lin} to the case of non-linear composites.

On the other hand, the stress admissibility condition (\ref{local:fct:space}.iii), i.e., $\tod{\tau}\in\mathscrbf{S}$, is handled as in the geometric variational principle (\ref{cost:fct:N}--\ref{princip:geo1}) through the minimization of the projection norm $\| \PSo \tod{\tau} \|_{\mathscrbf{H}_s}$. Likewise, the strain admissibility condition (\ref{local:fct:space}.i), i.e., $(\tod{\eta}-\macr{\tod{\eps}})\in\mathscrbf{E}_0$, is achieved by minimizing the norm $\| \PEzo(\tod{\eta}-\macr{\tod{\eps}})  \|_{\mathscrbf{H}_e}$ of the projection onto the subspace orthogonal to $\mathscrbf{E}_0$. According to the orthogonal decomposition \eqref{prop6}, then $\mathscrbf{E}_0^\bot=\mathscrbf{U}_e\oplus\mathscrbf{E}^\bot$ and $\PEzo(\tod{\eta}-\macr{\tod{\eps}})=\big(\toq{I}-\PEz\big)\tod{\eta}-\macr{\tod{\eps}}$, so that it amounts to the minimization of the functional $\tod{\eta} \mapsto \|\big(\toq{I}-\PEz\big)\tod{\eta}-\macr{\tod{\eps}} \|_{\mathscrbf{H}_e} $.

In summary, considering that none of the equations in \eqref{local:fct:space} is enforced exactly leads to the introduction of the cost functional $\mathcal{P}:\mathscrbf{H}_s\times\mathscrbf{H}_e\to\mathbb{R}$ defined as
\begin{equation}\label{cost:fct:P}
\mathcal{P}(\tod{\tau},\tod{\eta})=\Updelta{\textrm{Compat}}(\tod{\eta}) + \Updelta{\textrm{Const}}(\tod{\tau},\tod{\eta})  + \Updelta{\textrm{Equil}}(\tod{\tau}),
\end{equation}
with ``$\Updelta{\textrm{Compat}}$'', ``$\Updelta{\textrm{Const}}$'' and ``$\Updelta{\textrm{Equil}}$'' being error measures in kinematic, material and static admissibilities respectively. Based on the above, these terms are defined as: 
\begin{equation}\label{def:err:fct}
\Updelta{\textrm{Compat}}(\tod{\eta})=\frac{1}{2} \| \big(\toq{I}-\PEz\big)\tod{\eta}-\macr{\tod{\eps}} \|^2_{\mathscrbf{H}_e}, \qquad \Updelta{\textrm{Const}}(\tod{\tau},\tod{\eta})=\moy{r(\tou{x},\tod{\tau},\tod{\eta})}, \qquad \Updelta{\textrm{Equil}}(\tod{\tau})=\frac{1}{2} \| \PSo \tod{\tau} \|^2_{\mathscrbf{H}_s}.
\end{equation}
The stress and strain fields pair $(\tod{\sig},\tod{\eps})$ solution to \eqref{local:fct:space} are then characterized by the variational problem:
\begin{equation}\label{var:princip:P}
(\tod{\sig},\tod{\eps}) = \underset{ (\tod{\tau},\tod{\eta}) \in  \mathscrbf{H}_s\times\mathscrbf{H}_e }{\operatorname{arg\, min}}   \  \mathcal{P}(\tod{\tau},\tod{\eta}).
\end{equation}
Note that the stationary value of $\mathcal{P}$ is zero with each of its additive subparts being zero too at the solution. In the perspective of implementing a gradient-based minimization scheme for \eqref{var:princip:P} and as done previously, we now compute the gradient of the functional \eqref{cost:fct:P} in $\mathscrbf{H}_s\times\mathscrbf{H}_e$ endowed with the cross scalar product defined by \eqref{scalars} and \eqref{scalare}.

\begin{proposition}\label{prop:grad:P}
The partial gradients of $\mathcal{P}$ with respect to $\tod{\tau}$ and to $\tod{\eta}$ in $\mathscrbf{H}_s$ and $\mathscrbf{H}_e$ respectively, each being endowed with the associated energetic scalar product \eqref{scalars} and \eqref{scalare} are given by:
\[\left\{\begin{aligned}
&   \tou{\nabla}_{\!\tod{\tau}}\mathcal{P}(\tod{\tau},\tod{\eta}) =  \toq{L}_0 \big(\toq{L}^{-1}\tod{\tau} - \tod{\eta} \big) + \PSo \tod{\tau}, \\[1mm]
&  \tou{\nabla}_{\!\tod{\eta}}\mathcal{P}(\tod{\tau},\tod{\eta}) = \toq{L}_0^{-1} \big(\toq{L}\tod{\eta} - \tod{\tau}\big) + \big(\toq{I}-\PEz\big)\tod{\eta} - \macr{\tod{\eps}}.
\end{aligned}\right.\]
\end{proposition}

\begin{proof}
The partial gradient $\tou{\nabla}_{\!\tod{\tau}}\mathcal{P}(\tod{\tau},\tod{\eta})$ of $\mathcal{P}$ with respect to $\tod{\tau}\in\mathscrbf{H}_s$ is defined as the element of $\mathscrbf{H}_s$ that satisfies
\[
\pss{\tou{\nabla}_{\!\tod{\tau}}\mathcal{P}(\tod{\tau},\tod{\eta})}{\tilde{\tod{\tau}}} = \lim_{t\rightarrow 0} \frac{\mathcal{P}(\tod{\tau} + t \tilde{\tod{\tau}}, \tod{\eta} )-\mathcal{P}(\tod{\tau} , \tod{\eta} )}{t} \qquad \forall \tilde{\tod{\tau}}\in\mathscrbf{H}_s.
\]
According to \eqref{ecr:lin}, \eqref{cost:fct:P} and \eqref{def:err:fct} one finds
\[
\pss{\tou{\nabla}_{\!\tod{\tau}}\mathcal{P}(\tod{\tau},\tod{\eta})}{\tilde{\tod{\tau}}} = \moy{ \tilde{\tod{\tau}} : \toq{L}^{-1} \big(\tod{\tau} - \toq{L}\tod{\eta} \big) } + \pss{\PSo\tod{\tau}  }{\PSo\tilde{\tod{\tau}}  } \qquad \forall \tilde{\tod{\tau}}\in\mathscrbf{H}_s.
\]
Owing to the properties of the Green's operator $\toq{\Gamma}_0$, this identity can be recast as
\[
\pss{\tou{\nabla}_{\!\tod{\tau}}\mathcal{P}(\tod{\tau},\tod{\eta})}{\tilde{\tod{\tau}}} = \pss{  \toq{L}_0 \big(\toq{L}^{-1}\tod{\tau} - \tod{\eta} \big) + \PSo\tod{\tau}  }{\tilde{\tod{\tau}}  } \qquad \forall \tilde{\tod{\tau}}\in\mathscrbf{H}_s,
\]
which proves that
\[
\tou{\nabla}_{\!\tod{\tau}}\mathcal{P}(\tod{\tau},\tod{\eta}) =  \toq{L}_0 \big(\toq{L}^{-1}\tod{\tau} - \tod{\eta} \big) + \PSo \tod{\tau} .
\]
Similarly, one seeks $\tou{\nabla}_{\!\tod{\eta}}\mathcal{P}(\tod{\tau},\tod{\eta})$ as the element of $\mathscrbf{H}_e$ satisfying: 
\[
\pse{\tou{\nabla}_{\!\tod{\eta}}\mathcal{P}(\tod{\tau},\tod{\eta})}{\tilde{\tod{\eta}}} = \lim_{t\rightarrow 0} \frac{\mathcal{P}(\tod{\tau} , \tod{\eta}+ t \tilde{\tod{\eta}} )-\mathcal{P}(\tod{\tau} , \tod{\eta} )}{t} \qquad \forall \tilde{\tod{\eta}}\in\mathscrbf{H}_e.
\]
By definition, one has
\[
\pse{\tou{\nabla}_{\!\tod{\eta}}\mathcal{P}(\tod{\tau},\tod{\eta})}{\tilde{\tod{\eta}}} = - \moy{ \big(\tod{\tau} - \toq{L}\tod{\eta} \big):\tilde{\tod{\eta}} } + \pse{ \big(\toq{I}-\PEz\big)\tod{\eta} - \macr{\tod{\eps}}  }{\tilde{\tod{\eta}}} \qquad \forall \tilde{\tod{\eta}}\in\mathscrbf{H}_e,
\]
which can be rewritten as
\[
\pse{\tou{\nabla}_{\!\tod{\eta}}\mathcal{P}(\tod{\tau},\tod{\eta})}{\tilde{\tod{\eta}}} = \pse{ \toq{L}_0^{-1} \big(\toq{L}\tod{\eta} - \tod{\tau} \big) + \big(\toq{I}-\PEz\big)\tod{\eta} - \macr{\tod{\eps}}  }{\tilde{\tod{\eta}}} \qquad \forall \tilde{\tod{\eta}}\in\mathscrbf{H}_e,
\]
an identity which finally yields
\[
\tou{\nabla}_{\!\tod{\eta}}\mathcal{P}(\tod{\tau},\tod{\eta}) =  \toq{L}_0^{-1} \big(\toq{L}\tod{\eta} - \tod{\tau} \big) + \big(\toq{I}-\PEz\big)\tod{\eta} - \macr{\tod{\eps}}.
\]
\end{proof}

Proposition \ref{prop:grad:P} allows to characterize the solution to the variational problem (\ref{cost:fct:P}--\ref{var:princip:P}). The first-order optimality condition $\tou{\nabla}_{\!\tod{\tau}}\mathcal{P}(\tod{\sig},\tod{\eps}) = \tod{0}$ implies that $\toq{L}_0 (\toq{L}^{-1}\tod{\sig} - \tod{\eps} ) \in \mathscrbf{S}^\bot$. Therefore, according to the duality characterization \eqref{dual:charac} and the Riesz mapping \eqref{Riesz}, there exists $\tod{e}^*\in \mathscrbf{E}_0$ such that
\begin{equation}\label{opt:cond:int1}
(\toq{L}^{-1}\tod{\sig} - \tod{\eps} )=\tod{e}^*.
\end{equation}
 Considering the second optimality condition $\tou{\nabla}_{\!\tod{\eta}}\mathcal{P}(\tod{\sig},\tod{\eps}) = \tod{0}$ and applying to it the orthogonal projector $\PEz$ entail $\toq{\Gamma}_0(\toq{L}\tod{\eps} - \tod{\sig})=\tod{0}$. As a consequence, see \eqref{prop3}, there exists $\tod{s}\in \mathscrbf{S}$ such that
\begin{equation}\label{opt:cond:int2}
 (\toq{L}\tod{\eps} - \tod{\sig})=\tod{s}.
\end{equation}
Combining the equations \eqref{opt:cond:int1} and \eqref{opt:cond:int2} leads to $\toq{L}\tod{e}^*=-\tod{s}$. Applying the duality product with $\tod{e}^*$ to the previous equation implies
\[
\pdse{\toq{L}\tod{e}^*}{\tod{e}^*}=-\pdse{\tod{s}}{\tod{e}^*}=0,
\]
where the last equality follows from Lemma \ref{Hill:orth}. Then, the positive definiteness of the quadratic form $\pdse{\toq{L}\tod{e}^*}{\tod{e}^*}$ yields $\tod{e}^*=\tod{0}$. Hence, from \eqref{opt:cond:int1} one obtains $\tod{\sig}=\toq{L}\tod{\eps}$ which, inserted back in the optimality conditions, leads to $\PSo \tod{\sig}=\tod{0}$ and $(\toq{I}-\PEz)\tod{\eps}=\macr{\tod{\eps}}$. In summary, the optimality conditions associated with $\mathcal{P}$ lead to the following equations
\[
(\toq{I}-\toq{\Gamma}_0\toq{L}_0)\tod{\eps}=\macr{\tod{\eps}} , \qquad \tod{\sigma}(\tou{x})=\toq{L}(\tou{x})\tod{\eps}(\tou{x}) \text{ in }\mathcal{V} \qquad \text{and} \qquad \toq{\Gamma}_0\tod{\sig}=\tod{0},
\]
 which are equivalent to the original elasticity problem \eqref{local:fct:space} considered.\\

Finally, note that for the two-field geometric variational principle (\ref{cost:fct:P}--\ref{var:princip:P}), making use of the energetic principle \eqref{princip0} and based on Hill's lemma, the effective properties $\eff{\toq{L}}$ are computed using the identity:
\begin{equation}\label{Leff:var:P}
\eff{\toq{L}}\,\macr{\tod{\eps}}:\macr{\tod{\eps}}=\moy{\tod{\sig}(\tou{x})}\dc\moy{\tod{\eps}(\tou{x})}.
\end{equation}

\subsubsection{Connection with minimum energy principles}

\paragraph{Minimum energy principles for admissible fields.}

As the two-field variational principle (\ref{cost:fct:P}--\ref{var:princip:P}) has been motivated by geometric considerations, we would like to explore its connections with the conventional minimum energy principles in the case of linear constituents. The relationships between the strain-based principles have been explored in Proposition \ref{prop:equiv:J:N}, making use of the properties of the energy functional $\mathcal{J}$ in \eqref{cost:fct:J}. To do so for the two-field variational principle, we introduce the principle that is dual to \eqref{princip0}, i.e., the following stress-based minimum energy principle under controlled overall strain:
\begin{equation}\label{princip0:bis}
\tod{\sig}=\operatorname{arg}\min_{ \tod{s} \in \mathscrbf{S} }\mathcal{J}_c(\tod{s}) \quad \text{with} \quad \mathcal{J}_c(\tod{s}) = \frac{1}{2} \moy{\tod{s}(\tou{x}):\toq{L}^{-1}(\tou{x})\tod{s}(\tou{x})} - \moy{\tod{s}(\tou{x})}\dc \macr{\tod{\eps}}.
\end{equation}
As in Section \ref{sec:connect:en}, the gradient of $\mathcal{J}_c$ at $\tod{\sig}\in\mathscrbf{S}$ is defined as the element of $\mathscrbf{S}$ that satisfies
\[
\pss{\tou{\nabla}\mathcal{J}_c(\tod{s})}{\tilde{\tod{s}}}=\pss{ \toq{L}_0\big(\toq{L}^{-1}\tod{s}-\macr{\tod{\eps}}\big) }{ \tilde{\tod{s}} } \qquad \forall \tilde{\tod{s}} \in \mathscrbf{S}
\]
Therefore $\big( \tou{\nabla}\mathcal{J}_c(\tod{s}) - \toq{L}_0\big(\toq{L}^{-1}\tod{s}-\macr{\tod{\eps}}\big) \big)\in\mathscrbf{S}^\bot$ which, making use of the projector onto $\mathscrbf{S}$, yields:
\[
\tou{\nabla}\mathcal{J}_c(\tod{s}) = \toq{L}_0\big(\moy{\toq{L}^{-1}\tod{s}}-\macr{\tod{\eps}}\big) + \toq{\Delta}_0\toq{L}^{-1}\tod{s}.
\]
Proposition \ref{gen:decomp} can be used to obtain an equivalent but more convenient form as
\begin{equation}\label{grad:Jc}
\tou{\nabla}\mathcal{J}_c(\tod{s}) =  \toq{L}_0\big( \toq{I} - \PEz \big)\toq{L}^{-1}\tod{s} -  \toq{L}_0\macr{\tod{\eps}}.
\end{equation}
Reminding that, according to the Riesz mapping, one has for all $(\tod{\tau},\tod{\eta})\in\mathscrbf{H}_s\times\mathscrbf{H}_e$:
\begin{equation}\label{id:Riesz}
\| \big(\toq{I}-\PEz\big)\tod{\eta}-\macr{\tod{\eps}} \|_{\mathscrbf{H}_e}=\| \toq{L}_0\big(\toq{I}-\PEz\big)\tod{\eta}-\toq{L}_0\macr{\tod{\eps}} \|_{\mathscrbf{H}_s} \qquad \text{and} \qquad \| \PSo \tod{\tau} \|_{\mathscrbf{H}_s}=\| \toq{\Gamma}_0 \tod{\tau} \|_{\mathscrbf{H}_e},
\end{equation}
then combining the definitions \eqref{cost:fct:P} and \eqref{def:err:fct} with \eqref{grad:J}, \eqref{grad:Jc} and \eqref{id:Riesz} yields the result below that establishes a link between the geometric functional $\mathcal{P}$ and the energetic ones $\mathcal{J}$ and $\mathcal{J}_c$ for tensor fields that are admissible.

\begin{proposition}\label{equiv:P:J:Jc}
Considering a kinematically admissible strain field $\tod{e}\in\mathscrbf{E}$ such that $\moy{\tod{e}}=\macr{\tod{\eps}}$ and a statically admissible stress field $\tod{s}\in\mathscrbf{S}$, then along the following four ``trajectories'' it holds:
\[\begin{gathered}
 \mathcal{P}(\tod{s},\tod{e})=\moy{r(\tou{x},\tod{s},\tod{e})}, \qquad \mathcal{P}(\toq{L}\tod{e},\tod{e})=\frac{1}{2} \| \tou{\nabla}\mathcal{J} (\tod{e} - \macr{\tod{\eps}}) \|^2_{\mathscrbf{H}_e}, \qquad  \mathcal{P}(\tod{s},\toq{L}^{-1}\tod{s})=\frac{1}{2} \| \tou{\nabla}\mathcal{J}_c (\tod{s} ) \|^2_{\mathscrbf{H}_s},\\
\mathcal{P}(\toq{L}\tod{e},\toq{L}^{-1}\tod{s})=\frac{1}{2} \| \tou{\nabla}\mathcal{J}_c (\tod{s} ) \|^2_{\mathscrbf{H}_s} + \moy{r(\tou{x},\tod{s},\tod{e})} + \frac{1}{2} \| \tou{\nabla}\mathcal{J} (\tod{e} - \macr{\tod{\eps}}) \|^2_{\mathscrbf{H}_e}.
\end{gathered}\]
\end{proposition}

\paragraph{Comparison of minimization principles.} Proposition \ref{equiv:P:J:Jc} sheds light on the relationships existing between the energy functionals and $\mathcal{P}$ when the latter is employed with an admissible strain and/or an admissible stress field while assuming that the constitutive relations might be satisfied too. As the proposed geometric variational principle \eqref{var:princip:P} does not actually use any such constraint, we explore now its relationships with the minimum energy principles \eqref{princip0} and \eqref{princip0:bis} in general.

For short-hand notations, given $(\tod{\tau},\tod{\eta})\in \mathscrbf{H}_s\times\mathscrbf{H}_e$ and making use of Proposition~\ref{gen:decomp} we write:
\begin{equation}\label{decomp:tau:eta}
\left\{\begin{aligned}
& \tod{\tau} = \cS{\tod{\tau}} +  \cSo{\tod{\tau}} \quad \text{with} \quad \cS{\tod{\tau}} = \moy{\tod{\tau}}  + \PSz \tod{\tau}  \quad \text{and} \quad  \cSo{\tod{\tau}}=\PSo\tod{\tau}, \\
& \tod{\eta} = \cEz{\tod{\eta}} +  \cEzo{\tod{\eta}} \quad \text{with} \quad \cEz{\tod{\eta}} =  \PEz \tod{\eta}  \quad \text{and} \quad  \cEzo{\tod{\eta}}=\moy{\tod{\eta}} + \PEo\tod{\eta},
\end{aligned}\right.
\end{equation}
with the orthogonality in $\mathscrbf{H}_s$ and in $\mathscrbf{H}_e$ being understood in the sense of the corresponding energetic scalar products. Moreover, for the purpose of establishing the connection between the geometric and the energetic variational principles, let us recognize that the positive definite tensor $\toq{L}$ defines some energetic scalar products on the spaces of $\mathscrbf{L}^2_\text{per}\!\left(\mathcal{V}\right)$ \emph{strain} and \emph{stress} tensor fields, that we denote respectively as:
\begin{equation}
\pseL{\tod{\eta}}{\tod{\eta}}= \moy{\toq{L}(\tou{x})\tod{\eta}(\tou{x})\dc\tod{\eta}(\tou{x})} \qquad \text{and} \qquad \pssL{\tod{\tau}}{\tod{\tau}}= \moy{\tod{\tau}(\tou{x})\dc\toq{L}^{-1}(\tou{x})\tod{\tau}(\tou{x})} \label{scalarL}.
\end{equation}
In this context, the energetic scalar products \eqref{scalare} and \eqref{scalars} are temporarily denoted as 
\begin{equation}\label{scalarL0}
\pseLz{\tod{\eta}}{\tod{\eta}} \equiv \pse{\tod{\eta}}{\tod{\eta}} \qquad \text{and} \qquad \pssLz{\tod{\tau}}{\tod{\tau}} \equiv \pss{\tod{\tau}}{\tod{\tau}}
\end{equation}
to emphasize their dependence on the reference tensor $\toq{L}_0$ for the clarity of the exposition to come.

A tensor field in $\mathscrbf{H}_e$ can be seen as a strain field only if it belongs to the subspace $\mathscrbf{E}$, since it thus satisfies the kinematic admissibility condition. In order to extend the energetic functional $\mathcal{J}$ in \eqref{cost:fct:J} to any tensor field in $\mathscrbf{H}_e$ we evaluate the energy of its admissible component $\cEz{\tod{\eta}}\in\mathscrbf{E}_0$, based on \eqref{decomp:tau:eta}, by
\[
\mathcal{J}(\cEz{\tod{\eta}}) = \frac{1}{2} \pdse{\toq{L}(\macr{\tod{\eps}} + \cEz{\tod{\eta}})}{\macr{\tod{\eps}} + \cEz{\tod{\eta}}} = \frac{1}{2} \| \macr{\tod{\eps}} + \cEz{\tod{\eta}} \|^2_{\mathscrbf{H}_e,\toq{L}}
\]
where we make use of the norm associated with the scalar product \eqref{scalarL} defined by $\toq{L}$. Likewise, for a generic field $\tod{\tau}\in \mathscrbf{H}_s$, the complementary stress-based mechanical energy under the controlled overall strain $\macr{\tod{\eps}}$ is defined in terms of its statically admissible components $\cS{\tod{\tau}}\in\mathscrbf{S}$ as 
\[
\mathcal{J}_c(\cS{\tod{\tau}}) = \frac{1}{2} \pdse{\cS{\tod{\tau}}}{\toq{L}^{-1}\cS{\tod{\tau}}} - \moy{\cS{\tod{\tau}}}\dc \macr{\tod{\eps}} = \frac{1}{2} \| \cS{\tod{\tau}} \|^2_{\mathscrbf{H}_s,\toq{L}} - \moy{\cS{\tod{\tau}}}\dc \macr{\tod{\eps}},
\]
using \eqref{decomp:tau:eta} and the norm associated with \eqref{scalarL}. In this context, by a direct application of Lemma \ref{Hill:orth}, the evaluation of the averaged value of the error in constitutive relations \eqref{ecr:lin} for the admissible fields $\cS{\tod{\tau}}$ and $(\macr{\tod{\eps}}+\cEz{\tod{\eta}})$ yields
\[
\moy{r(\tou{x},\cS{\tod{\tau}},\macr{\tod{\eps}}+\cEz{\tod{\eta}})} = \frac{1}{2} \|  \cS{\tod{\tau}} - \toq{L}(\macr{\tod{\eps}} + \cEz{\tod{\eta}}) \|^2_{\mathscrbf{H}_s,\toq{L}}  = \frac{1}{2} \| \cS{\tod{\tau}} \|^2_{\mathscrbf{H}_s,\toq{L}} - \moy{\cS{\tod{\tau}}}\dc \macr{\tod{\eps}} + \frac{1}{2} \| \macr{\tod{\eps}} + \cEz{\tod{\eta}} \|^2_{\mathscrbf{H}_e,\toq{L}}.
\]
This shows that, for any admissible strain and stress fields $\cS{\tod{\tau}}\in\mathscrbf{S}$ and $\cEz{\tod{\eta}}\in\mathscrbf{E}_0$, the mean error in constitutive relations, which is positive, is equal to the associated total mechanical energy, i.e., 
\begin{equation}\label{tot:en:compat}
\moy{r(\tou{x},\cS{\tod{\tau}},\macr{\tod{\eps}}+\cEz{\tod{\eta}})} = \mathcal{J}_c(\cS{\tod{\tau}}) + \mathcal{J}(\cEz{\tod{\eta}}) \geq 0.
\end{equation}

Now, our aim is to show that, for any fields $(\tod{\tau},\tod{\eta})\in \mathscrbf{H}_s\times\mathscrbf{H}_e$, the geometric functional $\mathcal{P}(\tod{\tau},\tod{\eta})$ defined by \eqref{cost:fct:P} provides an upper bound on the total energy \eqref{tot:en:compat} of their admissible components $(\cS{\tod{\tau}},\cEz{\tod{\eta}})\in\mathscrbf{S}\times\mathscrbf{E}_0$. To do so, let us first rewrite \eqref{tot:en:compat} as
\[\begin{aligned}
\moy{r(\tou{x},\cS{\tod{\tau}},\macr{\tod{\eps}}+\cEz{\tod{\eta}})} & = \moy{r\big(\tou{x},\tod{\tau}-\cSo{\tod{\tau}},\tod{\eta}-(\cEzo{\tod{\eta}}-\macr{\tod{\eps}})\big)} \\
& = \frac{1}{2} \|   \big(\tod{\tau} - \toq{L}\tod{\eta}\big) -  \big(\cSo{\tod{\tau}} - \toq{L}(\cEzo{\tod{\eta}}-\macr{\tod{\eps}}) \big)    \|^2_{\mathscrbf{H}_s,\toq{L}},
\end{aligned}\]
so that
\begin{multline}\label{ecr:int}
\moy{r(\tou{x},\cS{\tod{\tau}},\macr{\tod{\eps}}+\cEz{\tod{\eta}})}  = \frac{1}{2} \| \tod{\tau} - \toq{L}\tod{\eta}  \|^2_{\mathscrbf{H}_s,\toq{L}} + \frac{1}{2} \|  \cSo{\tod{\tau}} - \toq{L}(\cEzo{\tod{\eta}}-\macr{\tod{\eps}})  \|^2_{\mathscrbf{H}_s,\toq{L}}  \\
- \pssL{ \tod{\tau} - \toq{L}\tod{\eta} }{ \cSo{\tod{\tau}} - \toq{L}(\cEzo{\tod{\eta}}-\macr{\tod{\eps}}) }.
\end{multline}
By Lemma \ref{Hill:orth} and Definition \eqref{scalarL} one has
\begin{equation}\label{ecr:int:1}
\|  \cSo{\tod{\tau}} - \toq{L}(\cEzo{\tod{\eta}}-\macr{\tod{\eps}})  \|^2_{\mathscrbf{H}_s,\toq{L}} =  \|  \cSo{\tod{\tau}}  \|^2_{\mathscrbf{H}_s,\toq{L}} + \|  \cEzo{\tod{\eta}}-\macr{\tod{\eps}}  \|^2_{\mathscrbf{H}_e,\toq{L}}.
\end{equation}
Moreover, the triangular inequality and the definition of the error in constitutive relations entail
\begin{equation}\label{ecr:int:2}\begin{aligned}
\big| \pssL{ \tod{\tau} - \toq{L}\tod{\eta} }{ \cSo{\tod{\tau}} - \toq{L}(\cEzo{\tod{\eta}}-\macr{\tod{\eps}}) } \big| & \leq \left( 2 \moy{r(\tou{x},\tod{\tau},\tod{\eta})}  \right)^{1/2} \left(  \|  \cSo{\tod{\tau}}  \|^2_{\mathscrbf{H}_s,\toq{L}} + \|  \cEzo{\tod{\eta}}-\macr{\tod{\eps}}  \|^2_{\mathscrbf{H}_e,\toq{L}}  \right)^{1/2} \\
& \leq \moy{r(\tou{x},\tod{\tau},\tod{\eta})}  + \frac{1}{2} \|  \cSo{\tod{\tau}}  \|^2_{\mathscrbf{H}_s,\toq{L}} + \frac{1}{2} \|  \cEzo{\tod{\eta}}-\macr{\tod{\eps}}  \|^2_{\mathscrbf{H}_e,\toq{L}} .
\end{aligned}\end{equation}
Therefore, using \eqref{ecr:int:1} and \eqref{ecr:int:2} in \eqref{ecr:int} leads to the following inequality:
\begin{equation}\label{bd:ecr:0}
\moy{r(\tou{x},\cS{\tod{\tau}},\macr{\tod{\eps}}+\cEz{\tod{\eta}})} \leq 2 \moy{r(\tou{x},\tod{\tau},\tod{\eta})} + \|  \cSo{\tod{\tau}}  \|^2_{\mathscrbf{H}_s,\toq{L}} + \|  \cEzo{\tod{\eta}}-\macr{\tod{\eps}}  \|^2_{\mathscrbf{H}_e,\toq{L}}.
\end{equation}
Now, in the definition of the geometric functional \eqref{cost:fct:P} with \eqref{def:err:fct} one can recognize that
\[
\| \PSo \tod{\tau} \|^2_{\mathscrbf{H}_s} = \|  \cSo{\tod{\tau}} \|^2_{\mathscrbf{H}_s,\toq{L}_0}   \qquad \text{and} \qquad \| \big(\toq{I}-\PEz\big)\tod{\eta}-\macr{\tod{\eps}} \|^2_{\mathscrbf{H}_e} = \| \cEzo{\tod{\eta}}-\macr{\tod{\eps}} \|^2_{\mathscrbf{H}_e,\toq{L}_0}, 
\]
using the notation \eqref{scalarL0} to highlight that the norms are associated with the scalar products defined by the reference tensor $\toq{L}_0$. Moreover, as the tensor $\toq{L}$ is such that its components are real-valued functions in $L^\infty\!\left(\mathcal{V} \right)$, the norms defined by $\toq{L}$ and $\toq{L}_0$ are equivalent. Therefore, there exists a constant $C>0$ that depends only on $\toq{L}$ and $\toq{L}_0$, such that
\[
\|  \cSo{\tod{\tau}}  \|^2_{\mathscrbf{H}_s,\toq{L}} + \|  \cEzo{\tod{\eta}}-\macr{\tod{\eps}}  \|^2_{\mathscrbf{H}_e,\toq{L}} \leq C\left( \|  \cSo{\tod{\tau}}  \|^2_{\mathscrbf{H}_s,\toq{L}_0} + \|  \cEzo{\tod{\eta}}-\macr{\tod{\eps}}  \|^2_{\mathscrbf{H}_e,\toq{L}_0}  \right).
\]
Using this inequality in \eqref{bd:ecr:0}, together with \eqref{tot:en:compat} and the definitions (\ref{cost:fct:P}--\ref{def:err:fct}) of the geometric cost functional yields the following result.

\begin{proposition}\label{prop:equiv:P:JJc}
There exists $c>0$, which depends only on $\toq{L}$ and $\toq{L}_0$, such that for all $(\tod{\tau},\tod{\eta})\in \mathscrbf{H}_s\times\mathscrbf{H}_e$:
\[
0 \leq \mathcal{J}_c(\cS{\tod{\tau}}) + \mathcal{J}(\cEz{\tod{\eta}}) \leq c\, \mathcal{P}(\tod{\tau},\tod{\eta}),
\] 
where $(\cS{\tod{\tau}},\cEz{\tod{\eta}})\in\mathscrbf{S}\times\mathscrbf{E}_0$ are the statically and kinematically admissible components of $(\tod{\tau},\tod{\eta})$.
\end{proposition}

This proposition implies that the overall minimization, with respect to $(\tod{\tau},\tod{\eta})\in \mathscrbf{H}_s\times\mathscrbf{H}_e$, of the geometric functional $\mathcal{P}$ to its null stationary value results in the overall minimization of the total mechanical energy associated with these fields. Indeed, the pair $(\tod{\sig},\tod{\eps})$ solution to \eqref{local:fct:space} satisfies $\mathcal{J}(\tod{\eps}-\moy{\tod{\eps}})=-J_c(\tod{\sig})$ with $\moy{\tod{\eps}}=\macr{\tod{\eps}}$, which is equivalent to having $\mathcal{P}(\tod{\sig},\tod{\eps})=0$.\\

\subsubsection{Non-linear composites}

As it can be done for energetic variational principles, we show how to extend the geometric variational principle (\ref{cost:fct:P}--\ref{var:princip:P}) to the case of non-linear constituents. To do so, consider a local energy density $w:\mathcal{V}\times\mathbb{R}^{d\times d}_{\operatorname{sym}} \to\mathbb{R}$, with $\mathbb{R}^{d\times d}_{\operatorname{sym}}$ being the space of symmetric second-order tensors. The potential $w$ is assumed to be convex in $\mathbb{R}^{d\times d}_{\operatorname{sym}}$ and weakly-coercive in the sense that for all $\tod{\eta},\tilde{\tod{\eta}} \in\mathbb{R}^{d\times d}_{\operatorname{sym}}$
\begin{equation}\label{def:weak:coer}
\text{if }\big(\partial_{\tod{\eta}} w(\tou{x},\tod{\eta}) - \partial_{\tod{\eta}} w(\tou{x},\tilde{\tod{\eta}}) \big):\big(\tod{\eta}-\tilde{\tod{\eta}}\big)=0 \text{ then } \tod{\eta}=\tilde{\tod{\eta}}.
\end{equation}
Its dual $w^*:\mathcal{V}\times\mathbb{R}^{d\times d}_{\operatorname{sym}}\to\mathbb{R}$, which is convex in $\mathbb{R}^{d\times d}_{\operatorname{sym}}$, is defined locally according to the classical Legendre-Fenchel transform: 
\begin{equation}\label{Fenchel}
w^*(\tou{x},\tod{\tau}) = \max_{\tod{\eta}\in\mathbb{R}^{d\times d}_{\operatorname{sym}}}\big\{ \tod{\tau}:\tod{\eta} - w(\tou{x},\tod{\eta})  \big\} \qquad \text{a.e. in }\mathcal{V},
\end{equation}
for all $\tod{\tau}\in\mathbb{R}^{d\times d}_{\operatorname{sym}}$. Note that, unlike in the linear case, the generic energy densities $w$ and $w^*$ are not necessarily quadratic.

The local error in constitutive relations functional \eqref{ecr:lin} in its form \eqref{ecr:lin:0} can then be generalized to non-linear constituents as was done in earlier studies, see \cite{Ladeveze} and the references therein. To do so, one defines:
\begin{equation}\label{ecr:nlin:0}
r(\tou{x},\tod{\tau},\tod{\eta}) = w(\tou{x},\tod{\eta}) + w^*(\tou{x},\tod{\tau}) - \tod{\tau}(\tou{x}):\tod{\eta}(\tou{x}) \qquad \forall (\tod{\tau},\tod{\eta})\in\mathscrbf{H}_s\times\mathscrbf{H}_e.
\end{equation}
As in the linear case, by definition of the Legendre-Fenchel transform one has $r(\tou{x},\tod{\tau},\tod{\eta})\geq 0$ in $\mathcal{V}$, while $r(\tou{x},\tod{\tau},\tod{\eta}) = 0$ locally if and only if $\tod{\tau}(\tou{x})=\partial_{\tod{\eta}} w(\tou{x},\tod{\eta})$. In this context, one considers the geometric variational principle (\ref{cost:fct:P}--\ref{var:princip:P}) with the term $\Updelta{\textrm{Const}}(\tod{\tau},\tod{\eta})$ being now defined using \eqref{ecr:nlin:0}. 

The partial gradients of the cost functional \eqref{cost:fct:P} that makes use of \eqref{ecr:nlin:0} can then be computed. First, the gradient $\tou{\nabla}_{\!\tod{\tau}}\mathcal{P}(\tod{\tau},\tod{\eta})$ is defined as the element of $\mathscrbf{H}_s$ that satisfies
\[
\pss{\tou{\nabla}_{\!\tod{\tau}}\mathcal{P}(\tod{\tau},\tod{\eta})}{\tilde{\tod{\tau}}}  = \moy{ \tilde{\tod{\tau}} :  \partial_{\tod{\tau}} w^*(\cdot,\tod{\tau})} - \moy{\tilde{\tod{\tau}}:\tod{\eta}} + \pss{\PSo\tod{\tau}  }{\tilde{\tod{\tau}}  } \qquad \forall \tilde{\tod{\tau}}\in\mathscrbf{H}_s.
\]
As this identity can be rewritten as
\[
\pss{\tou{\nabla}_{\!\tod{\tau}}\mathcal{P}(\tod{\tau},\tod{\eta})}{\tilde{\tod{\tau}}}  = \pss{ \toq{L}_0\big(\partial_{\tod{\tau}} w^*(\cdot,\tod{\tau}) - \tod{\eta}\big)  +\PSo\tod{\tau}}{\tilde{\tod{\tau}}} \qquad \forall \tilde{\tod{\tau}}\in\mathscrbf{H}_s,
\]
it can be deduced that
\begin{equation}\label{prop:grad:P:nl:tau}
\tou{\nabla}_{\!\tod{\tau}}\mathcal{P}(\tod{\tau},\tod{\eta}) =  \toq{L}_0\big(\partial_{\tod{\tau}} w^*(\cdot,\tod{\tau}) - \tod{\eta}\big)  + \PSo\tod{\tau}.
\end{equation}
Likewise, the partial gradient $\tou{\nabla}_{\!\tod{\eta}}\mathcal{P}(\tod{\tau},\tod{\eta})$ is defined as the element of $\mathscrbf{H}_e$ satisfying
\[
\pse{\tou{\nabla}_{\!\tod{\eta}}\mathcal{P}(\tod{\tau},\tod{\eta})}{\tilde{\tod{\eta}}} = \moy{\partial_{\tod{\eta}} w(\cdot,\tod{\eta}) : \tilde{\tod{\eta}}}  - \moy{ \tod{\tau}:\tilde{\tod{\eta}} } + \pse{ \big(\toq{I}-\PEz\big)\tod{\eta} - \macr{\tod{\eps}}  }{\tilde{\tod{\eta}}} \qquad \forall \tilde{\tod{\eta}}\in\mathscrbf{H}_e,
\]
from which one obtains finally:
\begin{equation}\label{prop:grad:P:nl:eta}
\tou{\nabla}_{\!\tod{\eta}}\mathcal{P}(\tod{\tau},\tod{\eta}) =  \toq{L}_0^{-1}\big(\partial_{\tod{\eta}} w(\cdot,\tod{\eta})  -  \tod{\tau} \big) + \big(\toq{I}-\PEz\big)\tod{\eta} - \macr{\tod{\eps}} .
\end{equation}

 The knowledge of the gradients \eqref{prop:grad:P:nl:tau} and \eqref{prop:grad:P:nl:eta} allows the implementation of gradient-based minimization algorithms. Moreover, as in the linear case, consider the associated optimality conditions and let $\tod{\eps}$, $\tod{\sig}$ denote the fields that satisfy $\tou{\nabla}_{\!\tod{\tau}}\mathcal{P}(\tod{\sig},\tod{\eps}) = \tod{0}$ and $\tou{\nabla}_{\!\tod{\eta}}\mathcal{P}(\tod{\sig},\tod{\eps}) = \tod{0}$. Based on the same arguments as those used in Section \ref{sec:min:proj:two:field}, these equations imply that there exist $\tod{e}^*\in\mathscrbf{E}_0$ and $\tod{s}\in\mathscrbf{S}$ such that
\begin{equation}\label{opt:cond:int1:nl}
 \partial_{\tod{\tau}} w^*(\cdot,\tod{\sig}) - \tod{\eps}=\tod{e}^* \qquad \text{and} \qquad \partial_{\tod{\eta}} w(\cdot,\tod{\eps})  -  \tod{\sig} = \tod{s}.
\end{equation}
 Inverting the relation $\partial_{\tod{\tau}} w^*(\cdot,\tod{\sig}) = \tod{\eps} + \tod{e}^*$ with the help of the Legendre-Fenchel transform \eqref{Fenchel} leads to $\tod{\sig}= \partial_{\tod{\eta}} w(\cdot,\tod{\eps} + \tod{e}^*)$, which inserted into the second equation in \eqref{opt:cond:int1:nl} implies
 \[
\partial_{\tod{\eta}} w(\cdot,\tod{\eps} + \tod{e}^*) - \partial_{\tod{\eta}} w(\cdot,\tod{\eps}) = - \tod{s}.
 \] 
On the one hand, applying the duality product with $\tod{e}^*$ to the above equation and using Lemma \ref{Hill:orth} entail
\begin{equation}\label{proof:nl:int1}
\pdse{ \partial_{\tod{\eta}} w(\cdot,\tod{\eps} + \tod{e}^*)  - \partial_{\tod{\eta}} w(\cdot,\tod{\eps}) }{\tod{e}^*}=-\pdse{\tod{s}}{\tod{e}^*}=0.
\end{equation}
On the other hand, owing to the convexity of $w$ in $\mathscrbf{H}_e$ it holds
\begin{equation}\label{proof:nl:int2}
 \big(  \partial_{\tod{\eta}} w(\tou{x},\tod{\eps} + \tod{e}^*) - \partial_{\tod{\eta}} w(\tou{x},\tod{\eps}) \big) : \tod{e}^*(\tou{x}) \geq 0 \qquad \text{a.e. in }\mathcal{V}.
\end{equation}
Therefore, \eqref{proof:nl:int1} and \eqref{proof:nl:int2} imply that
\[
 \big(  \partial_{\tod{\eta}} w(\tou{x},\tod{\eps} + \tod{e}^*) - \partial_{\tod{\eta}} w(\tou{x},\tod{\eps}) \big) : \tod{e}^*(\tou{x}) = 0 \qquad \text{a.e. in }\mathcal{V},
\]
which finally yields $\tod{e}^*(\tou{x})=\tod{0}$ as a consequence of the assumption \eqref{def:weak:coer}. In turn, this implies that $\tod{\sig}= \partial_{\tod{\eta}} w(\cdot,\tod{\eps})$ and thus, from \eqref{prop:grad:P:nl:tau} and \eqref{prop:grad:P:nl:eta}, one gets $\PSo\tod{\sig}=\tod{0}$ together with $(\toq{I}-\PEz)\tod{\eps} = \macr{\tod{\eps}}$. Therefore, the optimality conditions on $\mathcal{P}$ yield the equations:
\[
(\toq{I}-\toq{\Gamma}_0\toq{L}_0)\tod{\eps}=\macr{\tod{\eps}} , \qquad \tod{\sig}(\tou{x})= \partial_{\tod{\eta}} w(\tou{x},\tod{\eps}) \text{ in }\mathcal{V} \qquad \text{and} \qquad \toq{\Gamma}_0\tod{\sig}=\tod{0},
\]
which are equivalent to the original elasticity problem \eqref{local:fct:space} transposed to non-linear composites.

\section{Numerical implementation}\label{sec:num}

\subsection{Iterative minimization schemes}\label{sec:iter:schemes}

\subsubsection{Gradient-based algorithms}

Throughout this article, a number of variational principles have been investigated, based on energetic or geometric considerations. In this context, the aim of this section is to discuss the numerical implementation of some iterative minimization schemes. As the formulations considered involve either one-field or two-field cost functionals (see (\ref{cost:fct:J},\,\ref{cost:fct:N}) and \eqref{cost:fct:P} respectively), we consider in a generic setting the following variational problem
\begin{equation}\label{var:gen}
\tod{\chi} = \operatorname{arg}\underset{ \tilde{\tod{\chi}} \in \mathscrbf{H} }{\operatorname{min}}   \    \mathcal{T}(\tilde{\tod{\chi}}) \quad \text{with } \mathcal{T}:\,\tilde{\tod{\chi}}\in \mathscrbf{H}\mapsto \mathcal{T}(\tilde{\tod{\chi}})\in\mathbb{R},
\end{equation}
where the cost functional $\mathcal{T}$ is defined in a Hilbert space $\mathscrbf{H}$ that is either $\mathscrbf{H}_e$, $\mathscrbf{H}_s$ or $\mathscrbf{H}_s\times\mathscrbf{H}_e$, or a subset thereof, equipped with the corresponding energetic scalar product \eqref{scalare}, \eqref{scalars} or their cross product. The search for a minimizer of the variational problem \eqref{var:gen} can be performed by several descent algorithms. While our aim is not to optimize or discuss extensively the performances of such algorithms, reference can be made to \cite{Zeman,Brisard,Gelebart,Kabel} for conjugate-gradient based implementations and to \cite{Schneider} for a fast gradient method for computational homogeneization. Rather, we focus here on simple implementations and discuss their properties.\\

A first class of descent methods is obtained by choosing the gradient as the direction of descent:
\[
\tod{\chi}_{n+1} =  \tod{\chi}_{n} + \rho_{n}\, \tod{p}_{n} \qquad \text{with} \qquad \tod{p}_{n}= \tou{\nabla}\mathcal{T}(\tod{\chi}_{n}).
\]
One of the simplest descent algorithms along the gradient of $\mathcal{T}$ is obtained with a fixed step $\rho_n=-1$.

\begin{remark}
It is insightful to revisit the original fixed-point algorithm of \cite{Moulinec94,Moulinec} in the context of gradient-based minimization schemes. Indeed, as noted in \cite{Kabel}, the basic scheme introduced in \cite{Moulinec} can be interpreted as a gradient descent method with fixed step for the energetic functional $\mathcal{J}$ in \eqref{cost:fct:J}.
\end{remark}

Another descent method is the gradient method with optimal step, which involves the following line search at each step:
\begin{equation}\label{line:search}
\rho_{n}=\operatorname{arg}\underset{\rho\in\mathbb{R}}{\operatorname{min}}\,\mathcal{T}(\tod{\chi}_{n} + \rho\, \tod{p}_n) \qquad \text{where} \qquad \tod{p}_n= \tou{\nabla}\mathcal{T}(\tod{\chi}_{n}),
\end{equation}
which ensures that $\mathcal{T}$ decreases monotonically over iterations (a property which is not guaranteed by gradient methods with fixed step). In the case of linear constituents, the generic minimization problem \eqref{var:gen} is quadratic so that the gradient of $\mathcal{T}$ in $\mathscrbf{H}$ takes the form
\begin{equation}\label{grad:lin:gen}
\tou{\nabla}\mathcal{T}(\tilde{\tod{\chi}})=\toq{T}\tilde{\tod{\chi}} - \tod{t},
\end{equation}
where $\toq{T}$ is a linear operator from $\mathscrbf{H}$ into itself and $\tod{t}\in\mathscrbf{H}$, which may both be expressed in terms of the reference tensor $\toq{L}_0$ and the associated Green's operators $\toq{\Gamma}_0$ and $\toq{\Delta}_0$. From \eqref{grad:lin:gen}, the optimal step $\rho_n$ can be found analytically at each iteration by writing that 
\[
\psh{\tou{\nabla}\mathcal{T}(\tod{\chi}_{n} + \rho_n\, \tod{p}_n)}{\tod{p}_n}=0, \qquad \text{i.e., } \rho_n=-\frac{\| \tod{p}_n \|^2_{\mathscrbf{H}}}{\psh{\toq{T}\tod{p}_n}{\tod{p}_n}}.
\]
Note that the above expression could be further simplified given the explicit form of the linear operator $\toq{T}$ for specific cost functionals.

Alternatively, conjugate-gradient methods can be used to solve the variational problem \eqref{var:gen}. For linear materials, the Lippmann-Schwinger integral equation corresponding to the optimality conditions on $\mathcal{T}$ amounts in the linear system $\toq{T}\tilde{\tod{\chi}} = \tod{t}$ with the notations of Equation \eqref{grad:lin:gen}. As in the variational framework considered $\toq{T}$ is a self-adjoint operator from $\mathscrbf{H}$ into itself for the chosen scalar product, then the conjugate-gradient method may be used to solve this system. It reads as follows:\\[2mm]
\emph{Initialization:} choose $\tod{\chi}_0$ and compute 
\[
\tod{r}_0 = \tod{t}- \toq{T}\tod{\chi}_0 = - \tou{\nabla}\mathcal{T}(\tod{\chi}_0),\qquad  \tod{p}_0= \tod{r}_0.
\]
\emph{Then:} do $n=0,1, \dots$ until convergence
\[\begin{aligned}
& 1. &&  \alpha_n=\frac{\| \tod{r}_n \|^2_{\mathscrbf{H}}}{\psh{\toq{T}\tod{p}_n}{\tod{p}_n}} \\[1mm]
& 2. &&  \tod{\chi}_{n+1} =  \tod{\chi}_{n} + \alpha_{n}\, \tod{p}_{n} \\[1mm]
& 3. &&  \tod{r}_{n+1}=  \tod{r}_n - \alpha_n\, \toq{T}\tod{p}_n  \\[1mm]
& 4. &&  \beta_n=\frac{\|\tod{r}_{n+1}\|^2_\mathscrbf{H}}{\|\tod{r}_{n}\|^2_\mathscrbf{H}} \\[1mm]
& 5. &&  \tod{p}_{n+1}=  \tod{r}_{n+1} + \beta_n\, \tod{p}_n
\end{aligned} 
\]
Although this scheme shares the same \emph{structure} with those used in \cite{Zeman,Brisard,Gelebart}, there are subtle but fundamental differences between these algorithms. These differences are associated with the use of the energetic scalar products \eqref{scalare} and \eqref{scalars} together with the associated norms.

For non-linear materials, alternative conjugate-gradient methods can be used. The corresponding schemes maintain the previous structure but with $\tod{r}_n$ being defined as the steepest descent direction at each step, i.e., $\tod{r}_n=- \tou{\nabla}\mathcal{T}(\tod{\chi}_n)$, the parameter $\alpha_n$ being found by line search as in \eqref{line:search} and the parameter $\beta_n$ being chosen according to known formulae, which include the ones by Fletcher-Reeves, Polak-Ribiere, Hestenes-Stiefel and Dai-Yuan.

\begin{remark}\label{rmk:L0}
The gradient-based minimization schemes described previously depend all on the reference tensor $\toq{L}_0$, through the choice of the scalar product endowing $\mathscrbf{H}$. When assessing the convergence rate of the corresponding algorithms, not all scalar products are equivalent. Therefore, and as is well-known, $\toq{L}_0$ can be chosen so as to optimize the convergence rate, see the discussion in \cite{Moulinec18}. An optimization that frees itself from information about the microstructure has been proposed in \cite{Moulinec}. A quite natural question is whether it is possible to improve on the corresponding convergence rate when information on the microstructure is actually available and used. This issue is discussed further in \cite{Moulinec18} but it is beyond the scope of the present paper to investigate this question.
\end{remark}

\subsubsection{Stopping criteria}

In the approach by minimization based on either one of the iterative methods described previously, the expression of the gradient of $\mathcal{T}$ provides a sensible stopping criterion at the iterate $n$ as
\begin{equation}
\| \tou{\nabla}\mathcal{T} (\tod{\chi}_n) \|_\mathscrbf{H}  \leq \delta, 
\label{stop}
\end{equation}
 where the tolerance $\delta$  has to be chosen. For illustration purposes, let consider the energetic cost functional $\mathcal{J}:\mathscrbf{E}_0\subset\mathscrbf{H}_e\to\mathbb{R}$ defined by \eqref{cost:fct:J}, for which the criterion \eqref{stop} can be recast as
\begin{equation}
 \| \tou{\nabla}\mathcal{J} (\tod{e}^*_n) \|_{\mathscrbf{H}_e} = \| \toq{\Gamma}_0\tod{s}_n \|_{\mathscrbf{H}_e}\leq \delta \qquad \text{with} \qquad \tod{s}_n=\toq{L}(\macr{\tod{\eps}}+\tod{e}^*_n).
\label{stop:J}
\end{equation}
On the one hand, this criterion resembles the one suggested by \cite{Monchiet} with two differences: (i) the Green's operator used in the stopping criterion of \cite{Monchiet} is slightly different from the actual operator $\toq{\Gamma}_0$ associated with $\toq{L}_0$, and (ii) \cite{Monchiet} makes use of the $L^2$-norm of the operator, whereas \eqref{stop:J} involves the norm on $\mathscrbf{H}_e$ that is induced by the energetic scalar product \eqref{scalare}. On the other hand, the criterion \eqref{stop:J} differs from the original stopping criterion of \cite{Moulinec94,Moulinec}, which is based on the divergence of the stress field $\tod{s}_n$, i.e.,
\begin{equation}
 \| \operatorname{div}\tod{s}_n\|_{L^2} \leq \delta',  
\label{stop:sig}
\end{equation}
 where the $L^2$-norm of the divergence is computed in Fourier space using Parseval's theorem. Note that in \eqref{stop:J} the parameter $\delta$ can be chosen arbitrarily small, while by contrast $\delta'$ in \eqref{stop:sig} cannot. Moreover, using propositions \ref{grad:N} and \ref{prop:equiv:J:N}, a stopping criterion based on the norm $ \| \tou{\nabla}\mathcal{N} (\tod{e}^*_n) \|_{\mathscrbf{H}_e}$ can also be used.\\
 
The stopping criterion \eqref{stop} implies terminating the iterative minimization scheme once the gradient of the cost functional is sufficiently small. The value of the cost functional $\mathcal{T}$ itself is not relevant in this regard. In fact, for the energetic variational principles such as \eqref{princip0}, the stationary values of the corresponding cost functionals are not known beforehand. These values can actually be expressed in terms of the effective tensor $\eff{\toq{L}}$ which is rather computed a posteriori. On the contrary, for the geometric variational principles introduced in Section \ref{sec:geo:var:princip}, the cost functionals considered are defined so that their stationary values are zero. This allows to use the following stopping criteria: 
\[
\mathcal{N}(\tod{e}^*_n) \leq \delta'' \qquad \text{or} \qquad  \mathcal{P}(\tod{\tau}_n,\tod{\eta}_n)\leq \delta''
\]
where $\delta''$ can be chosen arbitrarily small at the continuous level. In the discretized versions of the proposed algorithms, the corresponding criteria must be adapted. Lastly, if one intends to compare the algorithm performances for different choices of the reference tensor $\toq{L}_0$ then these criteria must be normalized.

\subsection{Sample example: the Obnosov problem}

\subsubsection{FFT-based implementation}

The Green's operator $\toq{\Gamma}_0$ considered is a volume integral operator defined over the unit cell $\mathcal{V}$. While all previous developments are independent of the boundary conditions provided that they are compatible with Lemma \ref{Hill:orth}, for periodic media it is convenient to express this operator using the Fourier transform $\TF$, see Appendix \ref{math2}, as
\begin{equation}\label{int:op}
\toq{\Gamma}_0\tod{\tau}(\tou{x})=\TF^{-1}\left[\hat{\toq{\Gamma}}_0(\tod{\xi})\dc\TF[\tod{\tau}](\tou{\xi})\right]\!(\tou{x})\qquad \forall \tou{x}\in\mathcal{V},
\end{equation}
with the symmetric fourth-order tensor $\hat{\toq{\Gamma}}_0(\tod{\xi})$ being defined in closed-form in the Fourier space by
\[
\hat{\toq{\Gamma}}_0(\tod{0})=\toq{0} \quad \text{and} \quad \hat{\toq{\Gamma}}_0(\tod{\xi})=\left[\tod{\xi}\otimes\left(\tod{\xi}\cdot\toq{L}_0\cdot\tod{\xi}\right)^{-1}\otimes\tod{\xi}\right]_{\operatorname{sym}}\quad \forall \tod{\xi}\in\mathcal{R}^*\backslash\{\tou{0}\}.
\]
Owing to the convolution theorem, Eqn. \eqref{int:op} is the Fourier transform of a convolution, hence $\toq{\Gamma}_0$ is a non-local integral operator in the real space. This is also the case for the stress Green's tensor $\toq{\Delta}_0$. Note that, according to Proposition \ref{gen:decomp}, the tensor $\toq{\Delta}_0$ can be fully expressed in terms of $\toq{\Gamma}_0$. The Fourier-based formulation \eqref{int:op} is at the foundations of the FFT-based computational homogenization methods for periodic media and it is used in the numerical examples of this section.

In a typical numerical implementation, the unit cell $\mathcal{V}$ is discretized using a regular grid of $N^d$ pixels or voxels, which is used to sample data and unknowns. The discrete Fourier transform is computed by means of the FFT algorithm, using all the discrete frequencies associated with the discretization, see \cite{Moulinec}. The gradient-based minimization schemes of Section \ref{sec:iter:schemes} are implemented and, in accordance with the general principle of the FFT methods, the algorithms make use of \eqref{int:op} in order to compute the action of the operator $\toq{\Gamma}_0$ locally in the Fourier space, while actions of operators such as $\toq{L}(\tou{x})$ are computed locally in the real space. At convergence of these iterative schemes the accuracy of the solution is governed by the discretization, i.e., the grid size $N$.

\subsubsection{Comparison between minimization schemes}

\begin{figure}[htb]	
\centering
\includegraphics[height=0.2\textheight]{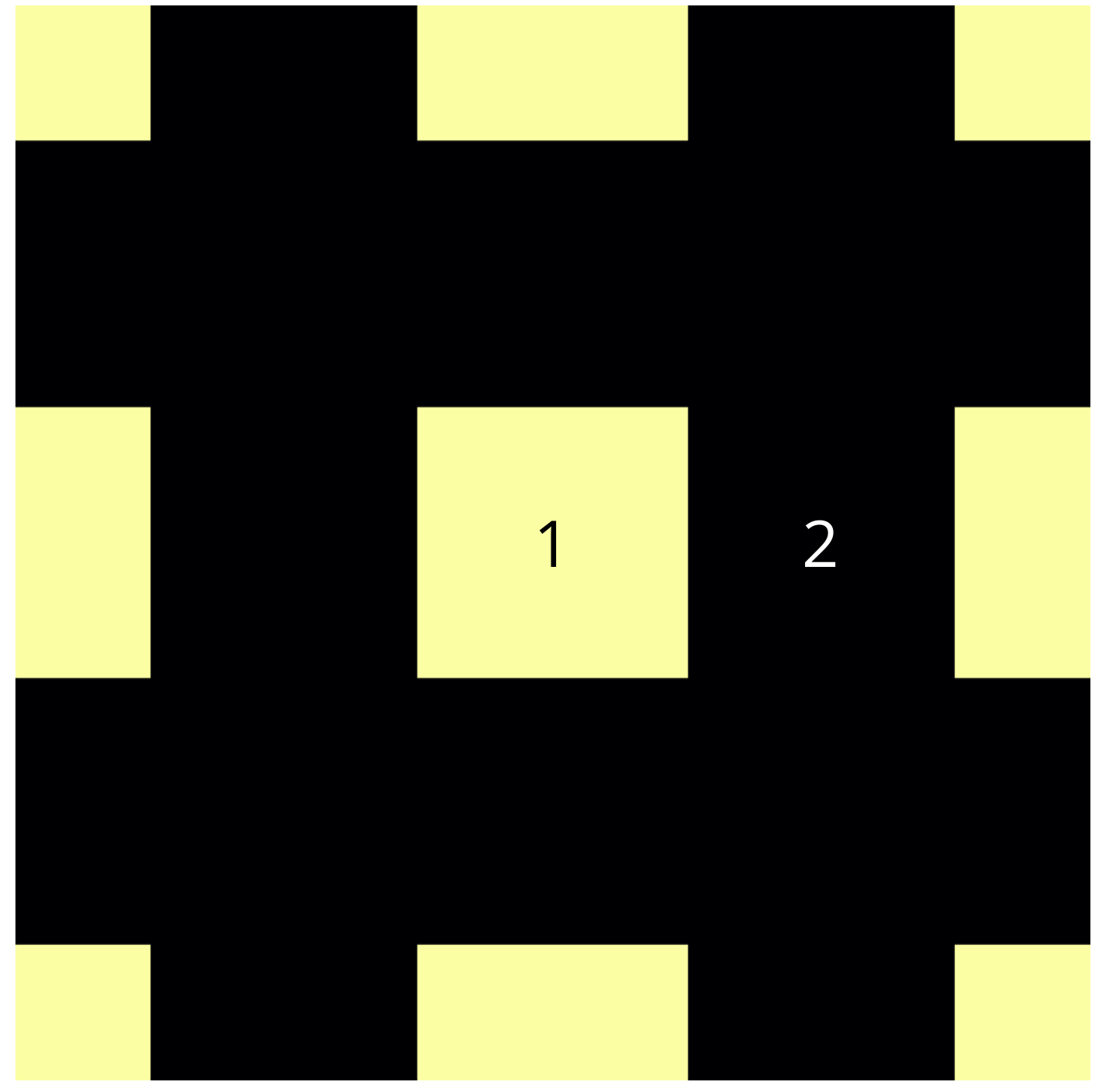}
\caption{Square inclusion in a square unit-cell with volume fraction $0.25$.}\label{Obnosov_config_per}
\end{figure}

In \cite{Obnosov} Y. Obnosov derived the effective conductivity of a square array of square inclusions with volume fraction $0.25$, see Figure \ref{Obnosov_config_per}. As the corresponding problem is equivalent to antiplane elasticity, let $\tod{L}(\tou{x})$ denote the second-order shear modulus matrix. Its effective value is isotropic according to:
\[
 \eff{\tod{L}}^{\operatorname{ex}}= \eff{L}^{\operatorname{ex}}\, \tod{I} \qquad \text{with} \qquad  \eff{L}^{\operatorname{ex}} =  {L}_{2} \sqrt{ \frac{1 + 3 z}{3  + z}}, \quad z = \frac{{L}_{1}}{{L}_{2}},
  \]
where the $L_{j}$ for $j=1,\,2$ are the isotropic shear moduli of the individual phases, the matrix being phase $2$. The contrast used in the subsequent simulations is $z=10^2$ with $L_2=1$ and the computational domain is an image of the unit-cell discretized into $512\!\times\!512$ pixels.

\begin{figure}[htb]	
\centering
\subfloat[Minimization of $\mathcal{J}$ in \eqref{cost:fct:J} and $\mathcal{N}$ in \eqref{cost:fct:N}.]{\includegraphics[height=0.26\textheight]{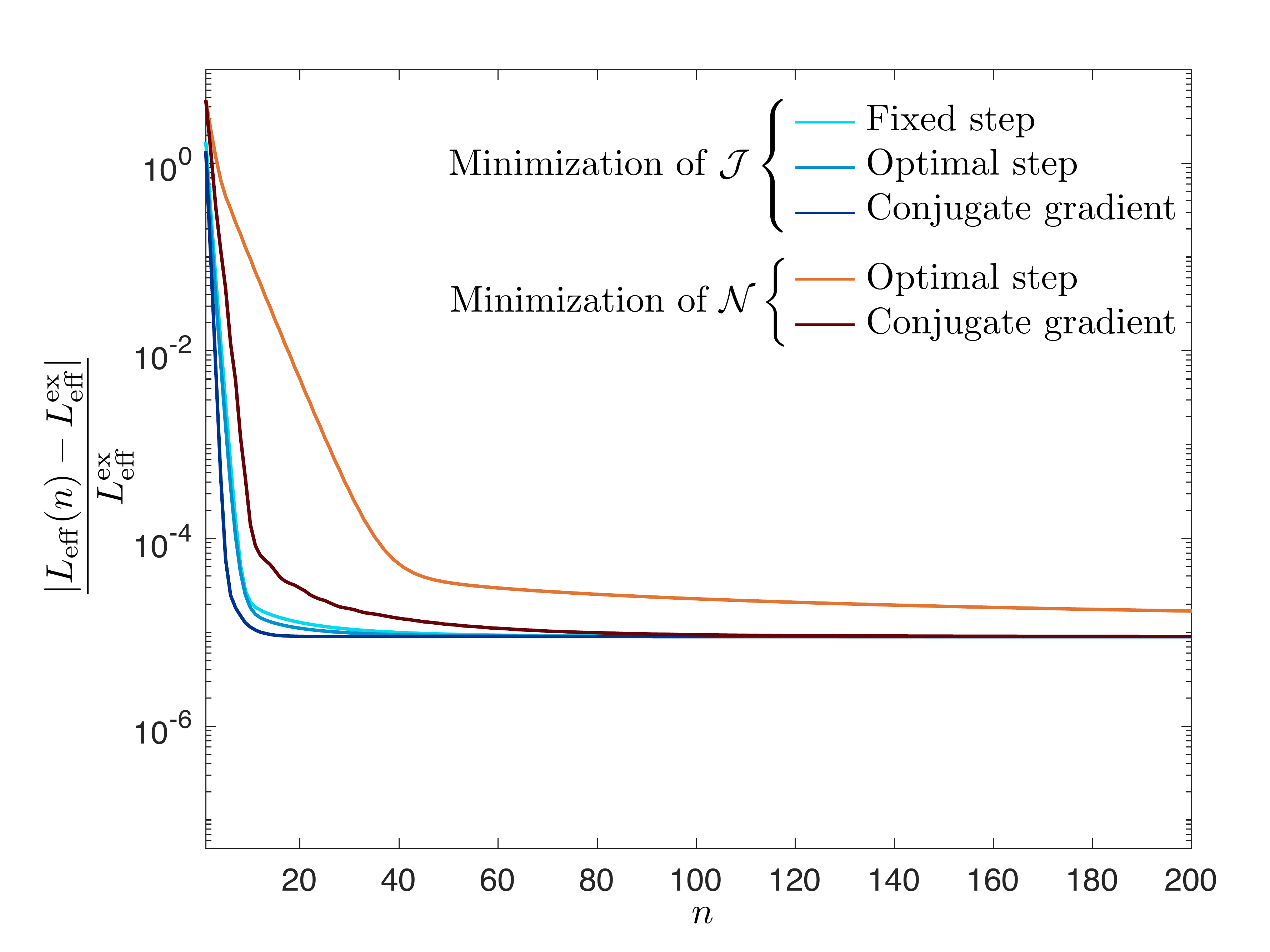}\label{Cvg_J_N_zeff}}
\subfloat[Minimization of $\mathcal{P}$ in \eqref{cost:fct:P}.]{\includegraphics[height=0.26\textheight]{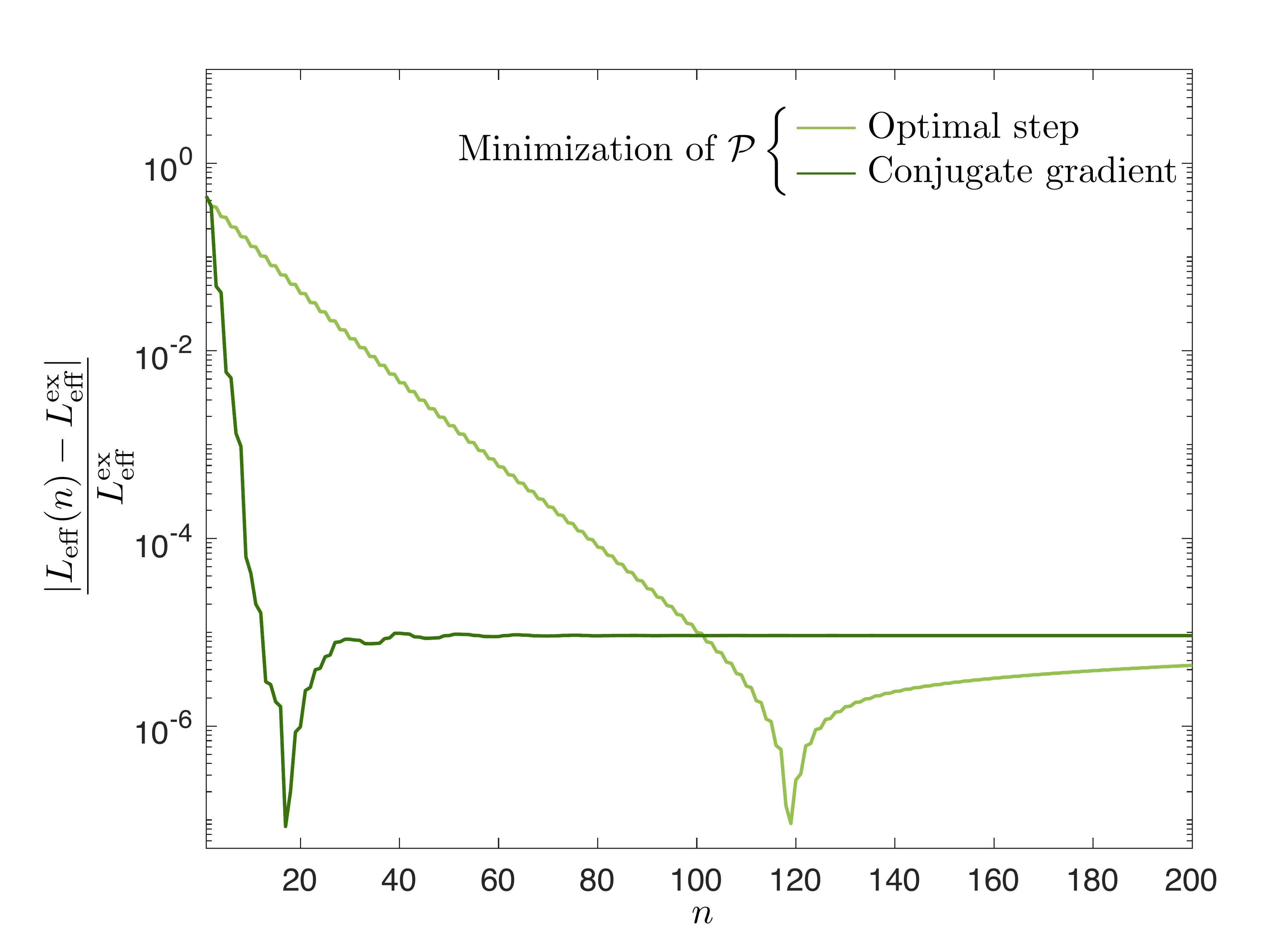}\label{Cvg_P_zeff}}
\caption{Comparison of gradient-based minimization schemes for the energetic and the geometric functionals considered: evolution of the relative error in effective property with respect to the number $n$ of iterations.}
\end{figure}

Computations were performed using the energetic cost functional \eqref{cost:fct:J} with the gradient method with fixed or optimal step and the conjugate-gradient method. Minimizations of the geometric functionals \eqref{cost:fct:N} and \eqref{cost:fct:P} are also implemented using the gradient method with optimal step and the conjugate-gradient method. For the strain-based schemes associated with $\mathcal{J}$ and $\mathcal{N}$, the corresponding reference medium was the optimal value derived in \cite{Michel} without information on the microstructure, i.e., $\tod{L}_0=L_0 \,\tod{I}$ with $L_{0}=\frac{z+1}{2}$. Moreover, they are initialized using $\tod{\chi}_0\equiv\tou{e}^*_0=\tod{0}$ with $\tod{\chi}_0$ being a first-order tensor in the antiplane elasticity problem considered. The effective property is computed at each iterate using $\eff{\toq{L}}\,\macr{\tod{\eps}}:\macr{\tod{\eps}}=\moy{\toq{L}(\tod{e}^*_n+\macr{\tod{\eps}}):(\tod{e}^*_n+\macr{\tod{\eps}})}$ according to \eqref{princip0}.

For the mixed schemes associated with $\mathcal{P}$ the reference medium was chosen as $\tod{L}_0=L_0 \,\tod{I}$ with $L_{0}=\sqrt{z}$. Alternative choices might be relevant, see Remark \ref{rmk:L0}, but we do not aim at investigating this issue further in the present study. The schemes are initialized using a pair of non-zero first-order tensors defined arbitrarily as $\tod{\chi}_0\equiv(\tou{\tau}_0,\tou{\eta}_0)=(\tou{1},\macr{\tou{\eps}})$. Note that, the functionals considered being convex, the gradient-based schemes are guaranteed to converge independently of the chosen starting point. Lastly, the effective modulus is computed according to \eqref{Leff:var:P}, i.e., using $\eff{\toq{L}}\,\macr{\tod{\eps}}:\macr{\tod{\eps}}=\moy{\tod{\tau}_n}\dc\moy{\tod{\eta}_n}$.

For the configuration investigated here, the conjugate-gradient is the fastest method. Moreover, for the minimization of the energetic function $\mathcal{J}$ very few performance differences are found between the gradient-based schemes and the so-called basic scheme of \cite{Moulinec}. In terms of computation of the effective property, the minimization of the geometric functionals $\mathcal{N}$ using the conjugate-gradient method yields comparable performances while the optimal-step implementation is associated with a slower convergence rate, see Figure \ref{Cvg_J_N_zeff}. Similar conclusions are found for the minimization of the geometric functional $\mathcal{P}$ as shown Fig. \ref{Cvg_P_zeff}. All methods compute the effective property with the same precision with a relative error of about $9\cdot10^{-6}$ measured at convergence. The specific values of the error reached at convergence for such algorithms is specific to (i) the physical configuration considered, (ii) the method implementation and (iii) the discretization used, see the discussion in \cite{Moulinec18}.

\begin{figure}[ht]	
\centering
\subfloat[Minimization of the geometric functional $\mathcal{N}$.]{\includegraphics[height=0.26\textheight]{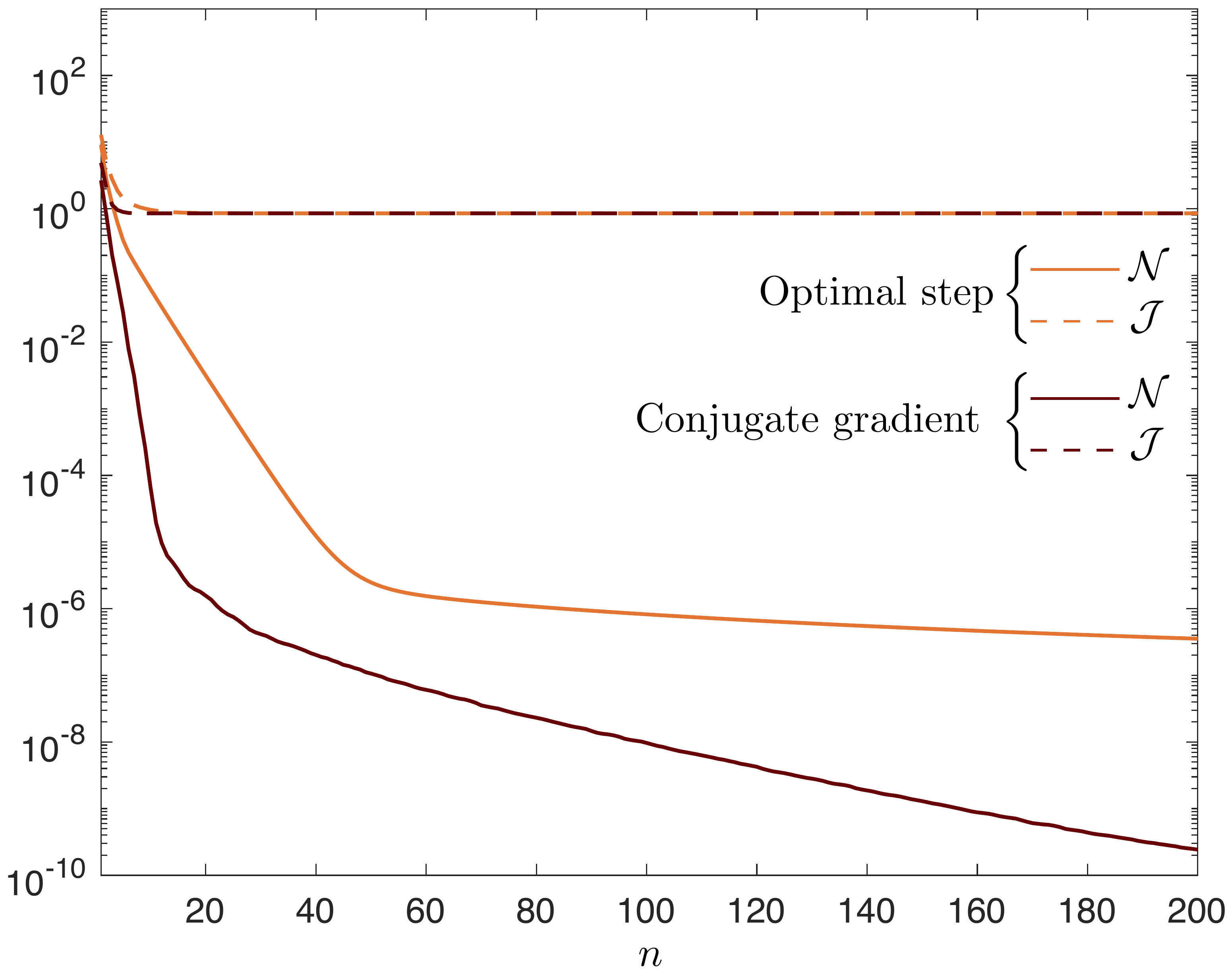}\label{Cvg_J_N}}
\subfloat[Minimization of the geometric functional $\mathcal{P}$.]{\includegraphics[height=0.26\textheight]{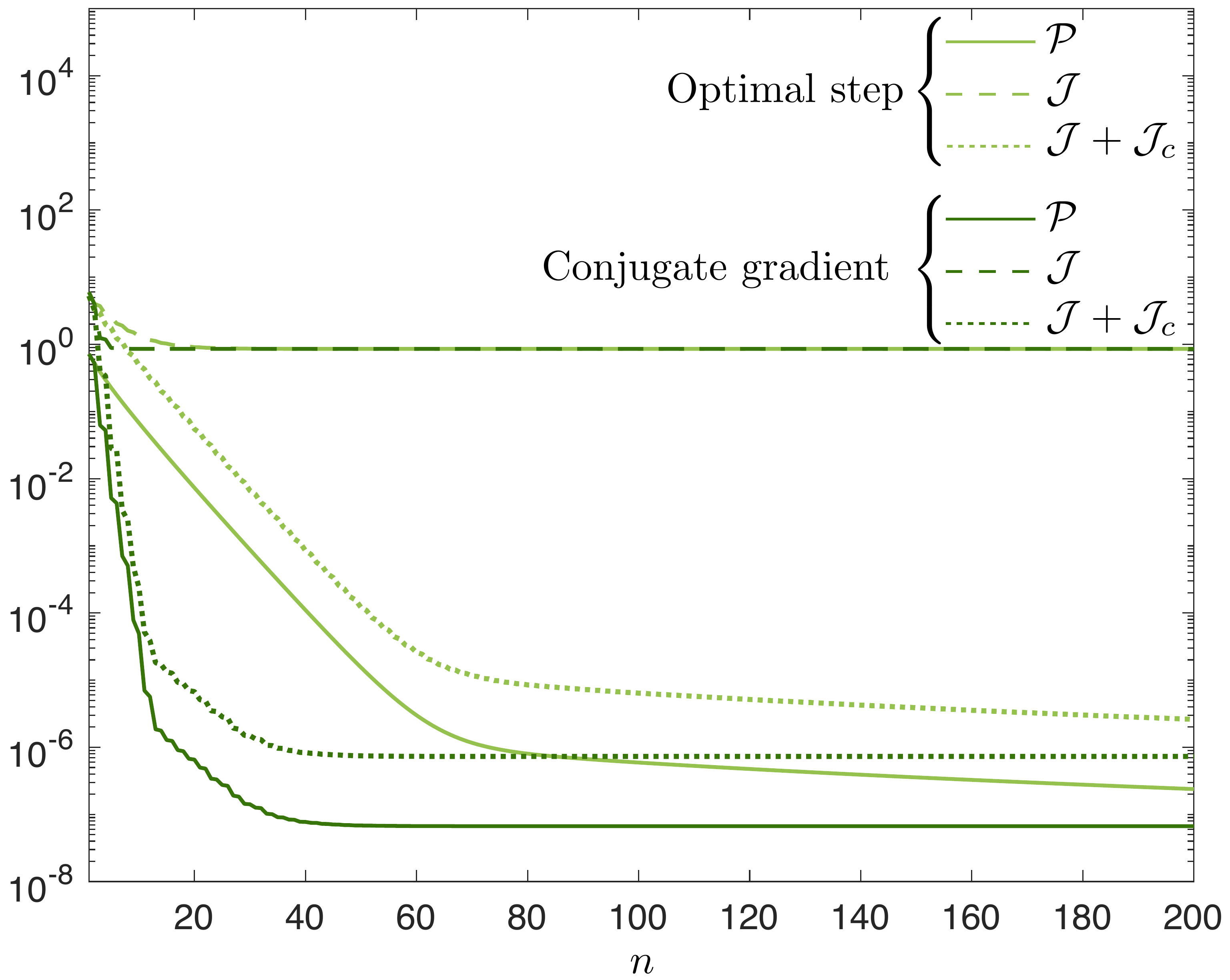}\label{Cvg_P}}
\caption{Comparison of gradient-based minimization schemes for the geometric functionals $\mathcal{N}$ and $\mathcal{P}$: evolution of the functionals with respect to the number $n$ of iterations (solid curves). For each scheme, the mechanical energy, quantified by the functionals $\mathcal{J}$ and $\mathcal{J}_c$, of the \emph{admissible components} at each iterate is plotted alongside (dashed and dotted curves).} \label{Cvg:geo:fct}
\end{figure}

For the proposed geometric variational principles, Figure \ref{Cvg:geo:fct} highlights that the functionals $\mathcal{N}$ and $\mathcal{P}$ are minimized monotonically as expected. Moreover, the mechanical energy, as quantified by the energetic functionals $\mathcal{J}$ and $\mathcal{J}_c$, for the \emph{admissible components} of the fields at each iterate decrease in accordance with the propositions \ref{prop:equiv:J:N} and \ref{prop:equiv:P:JJc} with $\mathcal{J}$ converging to an $\mathcal{O}(1)$ value that is equal to the effective energy, while $(\mathcal{J}+\mathcal{J}_c)$ converges to zero. In Fig. \ref{Cvg_J_N} and \ref{Cvg_P} it can also be seen that these energy terms are minimized monotonically.

As a conclusion of this section, the comparison of the proposed approach with earlier methods, i.e., the minimization of the geometric functionals compared to the energetic ones, is not characterized by an improvements of the numerical performances, in terms of, e.g., accuracy or speed of convergence. Rather, what should be retained from this comparison is that the proposed method is conceptually different in that it treats the three equations of the local problem on an equal footing. In doing so, its advantage lies in the fact that it make it easy to deal with situations where the constitutive model is partially or fully unknown, see \cite{Bonnet,Kirchdoerfer,Nouy,Staber}.

\subsubsection{Visualization of cost functional minimizations}

\begin{figure}[th]	
\centering
\subfloat[Gradient method with optimal step.]{\includegraphics[width=0.45\textwidth]{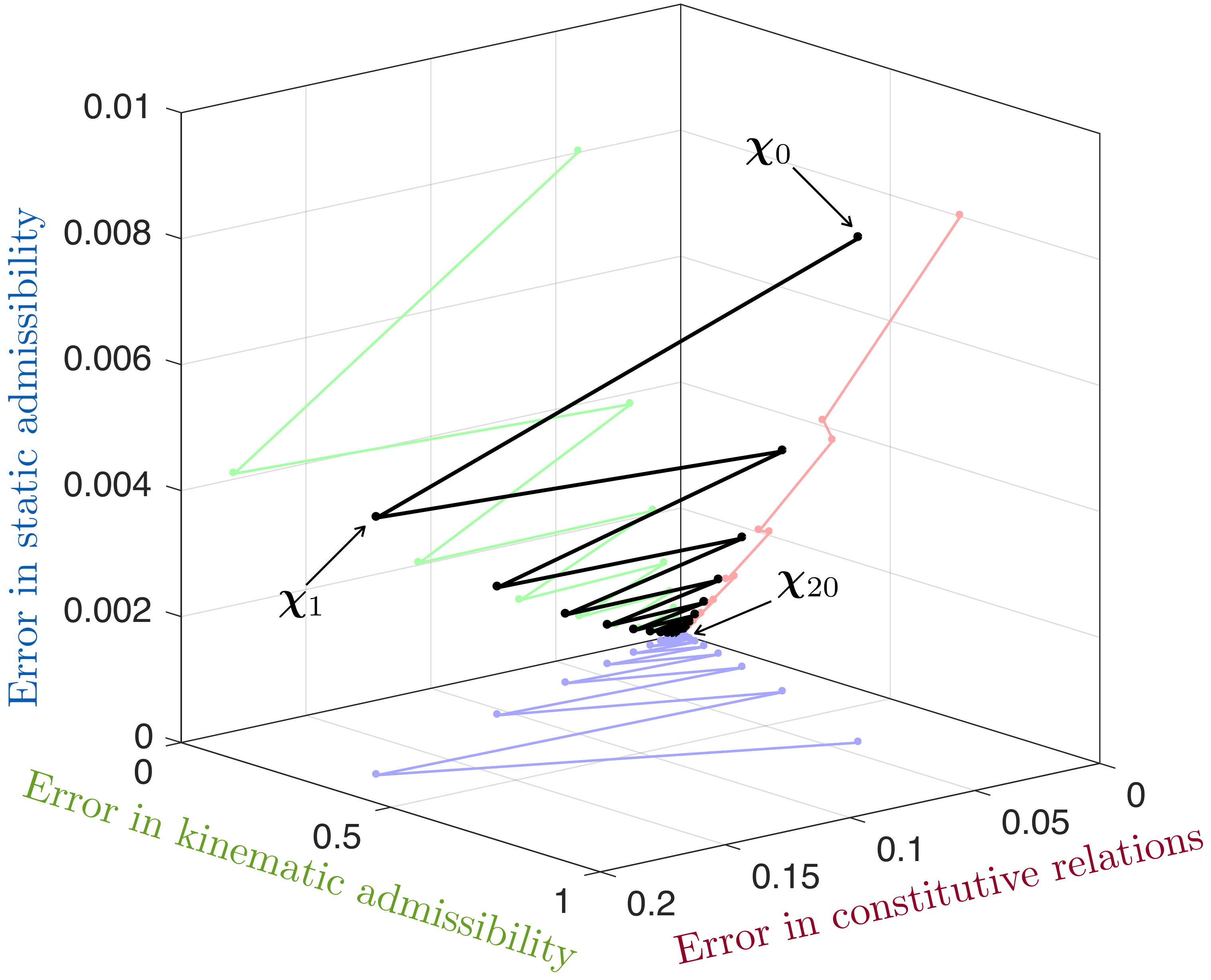}\label{minP_visu_opt}}\hspace{4mm}
\subfloat[Conjugate-gradient.]{\includegraphics[width=0.45\textwidth]{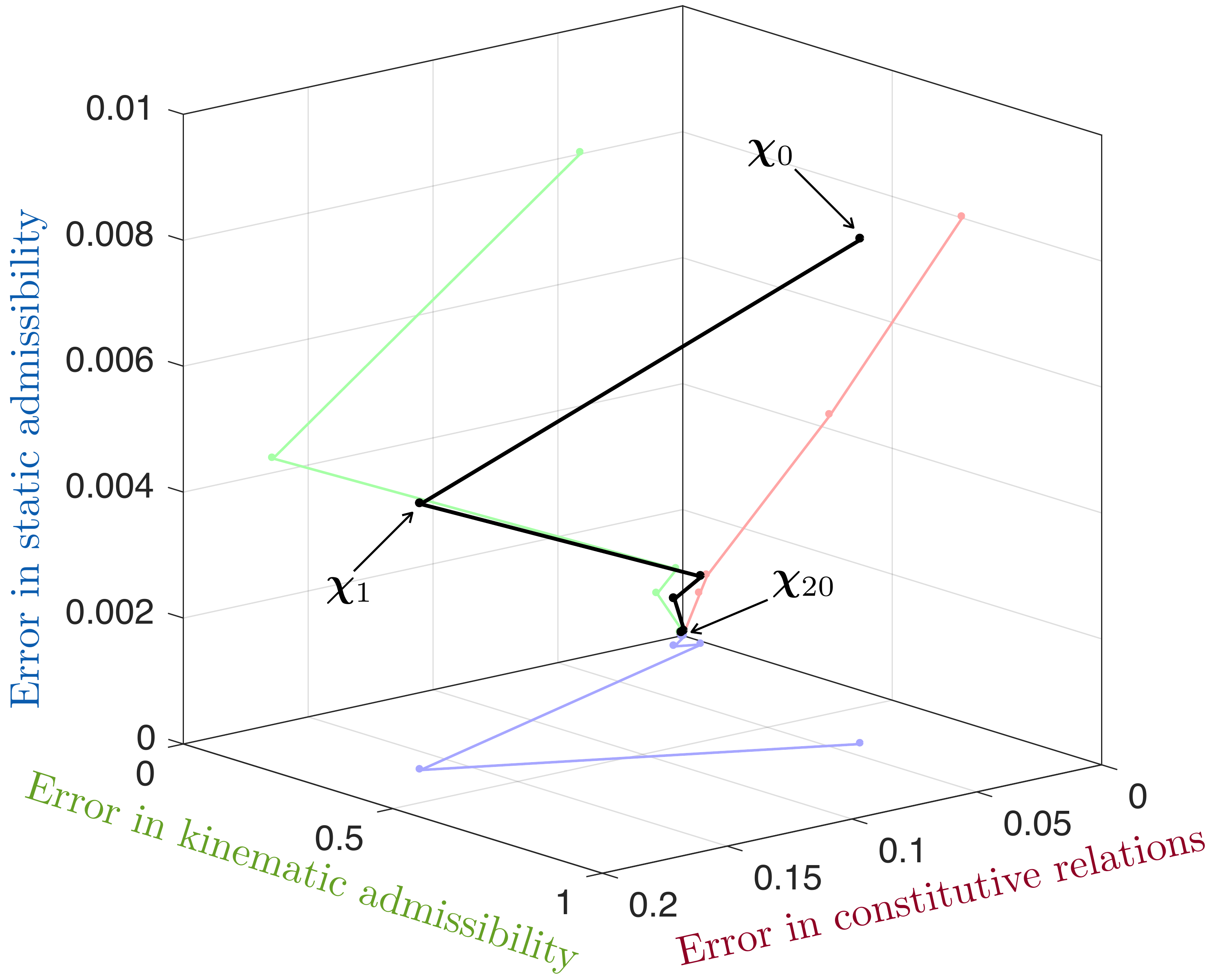}\label{minP_visu_CG}}
\caption{Visualization of the iterated solution in a projection space indicating the deviation from the validity of each of the equations in \eqref{local:fct:space}: gradient-based implementations of the proposed two-field geometric variational principle (\ref{cost:fct:P}--\ref{var:princip:P}) based on the functional $\mathcal{P}$.} \label{rep:solutions2}
\end{figure}

For the Obnosov problem described previously, Figure \ref{rep:solutions2} represents the solutions computed at each iterate for the gradient-based implementations of the proposed geometric variational principle (\ref{cost:fct:P}--\ref{var:princip:P}) using a projection space quantifying the deviation from the validity of each of the equations of the problem \eqref{local:fct:space}, see Fig. \ref{Fig:Schematics}. More precisely, for each iterate the solution is represented as a point with the following coordinates:
\[
\tou{\chi}_n\equiv(\tod{\tau}_n,\tod{\eta}_n) \rightarrow \big(x_n,y_n,z_n\big)=\big( \Updelta{\textrm{Const}}(\tod{\tau}_n,\tod{\eta}_n) , \Updelta{\textrm{Compat}}(\tod{\eta}_n) , \Updelta{\textrm{Equil}}(\tod{\tau}_n) \big).
\]
with the functionals ``$\Updelta{\textrm{Compat}}$'', ``$\Updelta{\textrm{Const}}$'' and ``$\Updelta{\textrm{Equil}}$'' being defined by \eqref{def:err:fct}.

\begin{figure}[htb]	
\centering
\subfloat[Basic-scheme of \cite{Moulinec}.]{\includegraphics[width=0.33\textwidth]{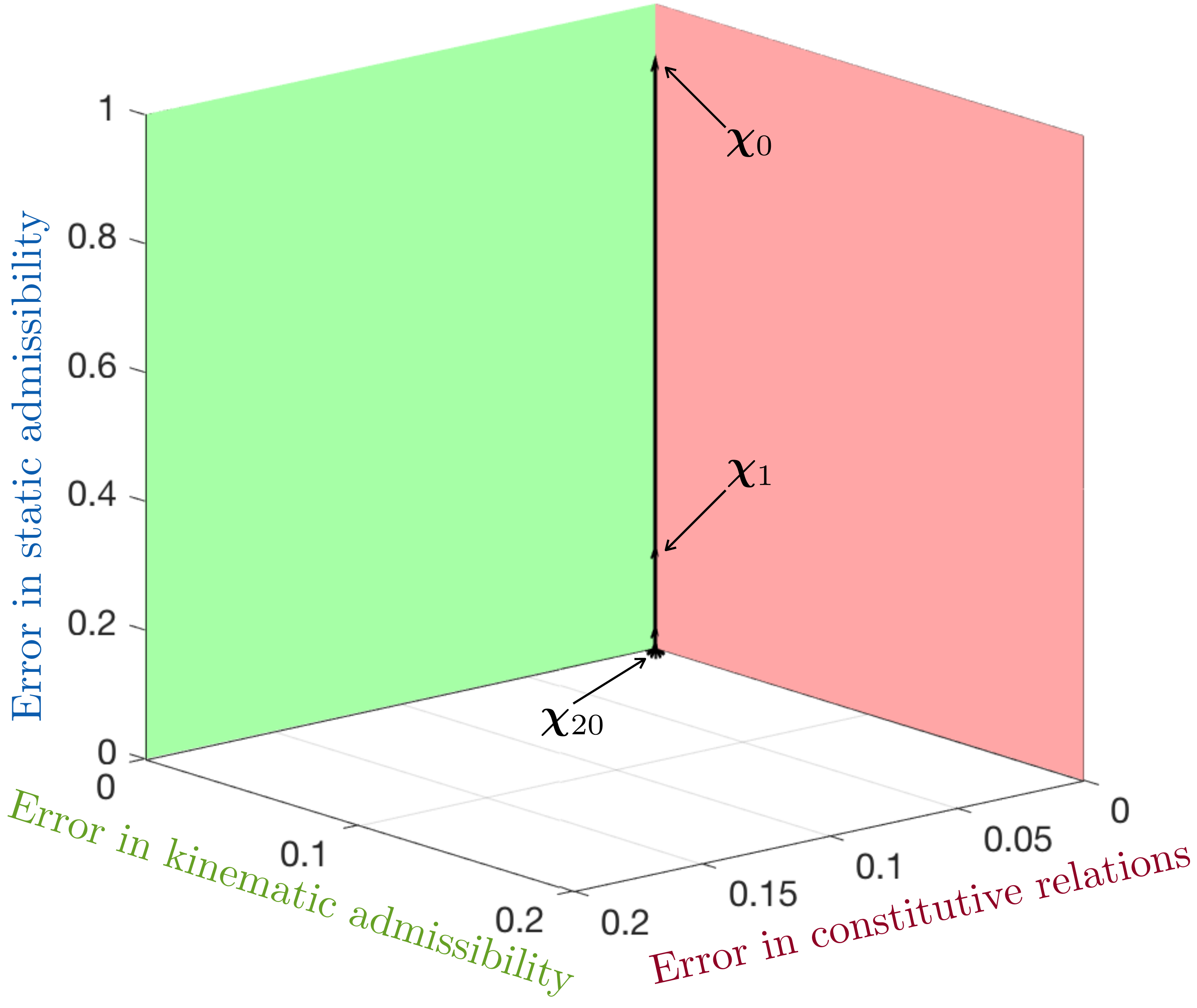}\label{minP_visu_MS}}
\subfloat[Eyre-Milton scheme in \cite{Eyre}.]{\includegraphics[width=0.33\textwidth]{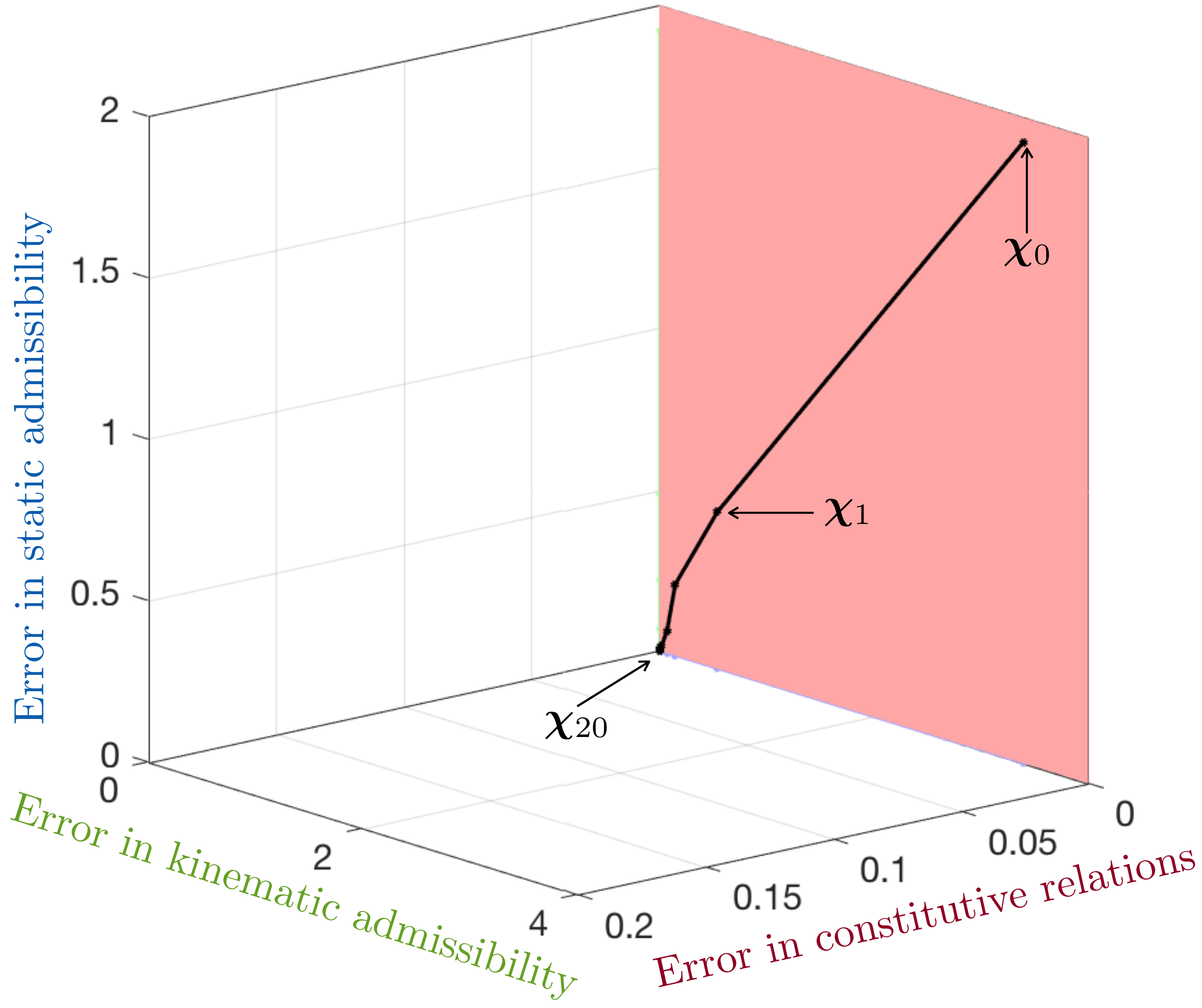}\label{minP_visu_EM}}
\subfloat[Augmented-Lagrangian scheme \cite{Michel}.]{\includegraphics[width=0.33\textwidth]{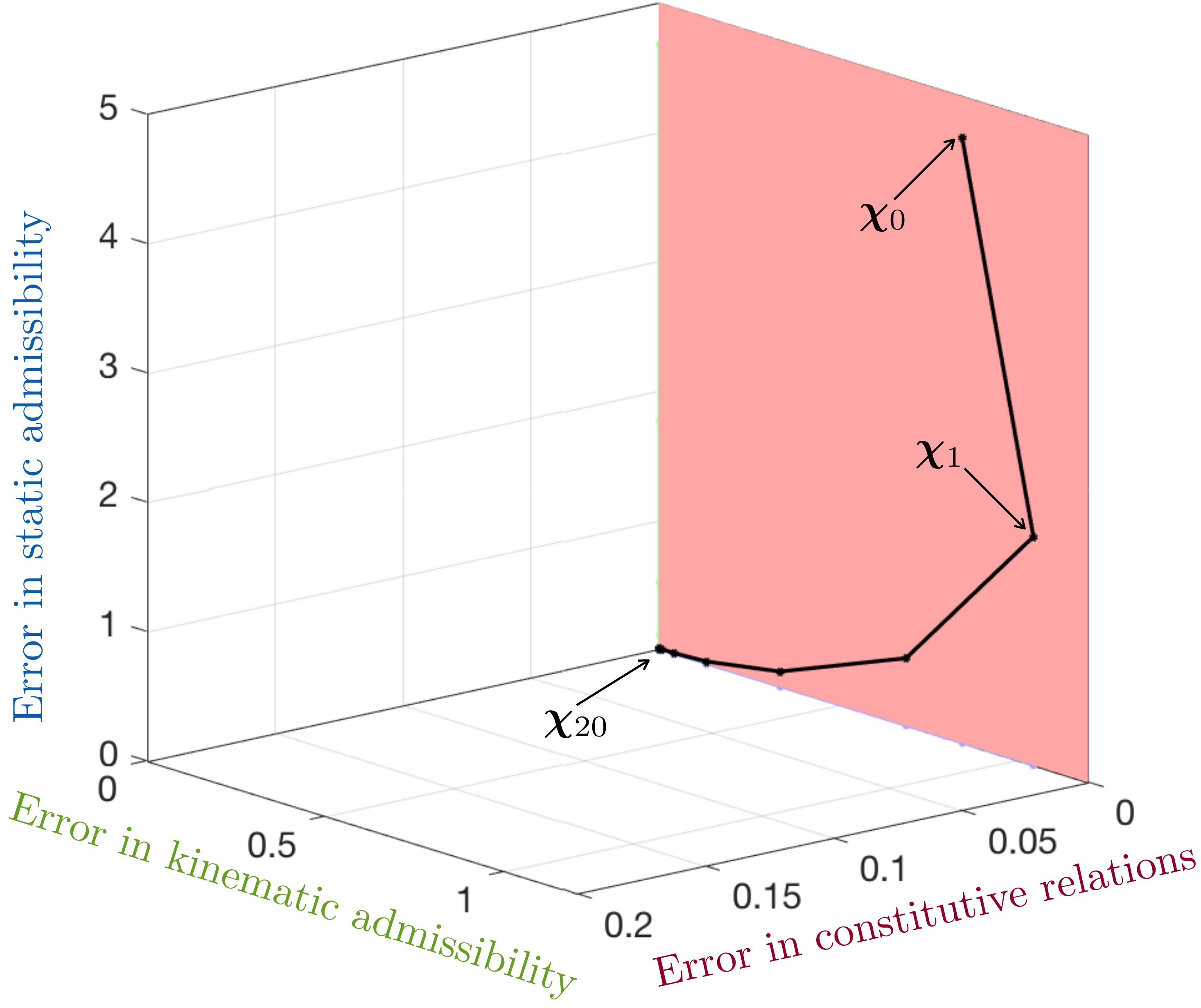}\label{minP_visu_AL}}
\caption{Visualization of the iterated solution for classical schemes in a projection space indicating the deviation from the validity of each of the equations in \eqref{local:fct:space}. The green and red planes are placed at the origins of the corresponding axes and displayed when the iterated solution evolves within them.} \label{rep:solutions}
\end{figure}

For comparison, three classical methods are also implemented and their corresponding solutions represented in the same projection space in Figure \ref{rep:solutions}. The iterates of the so-called basic scheme of \cite{Moulinec} are shown in Fig. \ref{minP_visu_MS} using the coordinates
\[
\tou{\chi}_n\equiv \tod{e}^*_n \rightarrow \big(x_n,y_n,z_n\big)=\big( 0 , 0 , \Updelta{\textrm{Equil}}(\toq{L}(\tod{e}^*_n+\macr{\tod{\eps}})) \big)
\]
since, by definition, this scheme satisfies by construction both the kinematic admissibility and the constitutive relations. The static admissibility is met only at convergence. The Eyre-Milton scheme of \cite{Eyre} and the augmented-Lagrangian method of \cite{Michel} are also implemented using the common single-field formalism of \cite{Silva} and their iterates are represented using the coordinates
\[
\tou{\chi}_n\equiv \tod{e}_n \rightarrow \big(x_n,y_n,z_n\big)=\big( 0 , \Updelta{\textrm{Compat}}(\tod{e}_n) , \Updelta{\textrm{Equil}}(\toq{L}\tod{e}_n )\big)
\]
as they only enforce the constitutive relations by construction. Kinematic and static admissibilities are satisfied at convergence. The figures \ref{rep:solutions2} and \ref{rep:solutions} allow to visualize the behavior of the iterated solution for each algorithm and to compare their convergence performances. For example, Figure \ref{rep:solutions2} sheds a new light on the behavior of the proposed algorithms compared to the only information provided by the computation of the effective property, see Fig. \ref{Cvg_P_zeff}.

%%%%%%%%%%%%%%%%%%%%%%%%%%%%

\section{Conclusions}

Considering an arbitrary periodic composite subjected to an applied macroscopic strain, the associated local mechanical problem is recast in a geometric formalism that relies on the definition of spaces of kinematically and statically admissible tensor fields. These are Hilbert spaces associated with two different energetic scalar products that are formulated using a reference and uniform elasticity tensor. The corresponding strain and stress Green's operators are considered and their geometric properties are investigated. In particular, these operators are shown to generate orthogonal decompositions of second-order tensors fields. In this context, two geometric variational principles are proposed to compute the solution to the mechanical problem considered, as well as the effective properties of the composite. First, a strain-based variational principle is proposed through the introduction of a functional that measures the lack of static admissibility of an associated stress test field. The gradient of this functional is computed and the connection with the classical energetic variational principle is discussed. Building from this preliminary idea, a two-field variational principle is proposed. Its aim is to relax all of the equations of the local mechanical problem so as to reach the solution through an unconstrained minimization process. Doing so, one of our objectives is to enable the treatment of problems where the constitutive relations are partially unknown or uncertain. The functional partial gradients are computed and some connections with minimum energy principles are discussed. With these geometric functionals at hand, their minimization is addressed using gradient-based iterative minimization schemes. A gradient descent scheme with optimal step and the conjugate-gradient method are both implemented and their performances are illustrated on the prototypical Obnosov problem for which analytical solutions are available. Using a FFT-based implementation, the proposed schemes are confronted to the standard approaches that revolve around the classical minimum energy principles. Lastly, the geometric setting considered allows to visualize the evolutions of the iterated solutions in a 3D system of coordinates corresponding to the kinematic, static and material admissibility conditions, which is illustrated for a number of schemes.

This study lays the groundwork for the homogenization of composites whose constitutive properties are partially unknown and this will be the subject of future works. Moreover, there are open questions regarding the connections between the energetic and the geometric variational principles, in particular to assess whether and under which conditions the latter yield some minimization principles for the former. Lastly, in the context considered, let us mention as a perspective that there exist other algorithms pertaining to convex optimization, such as the Alternating Direction Method of Multipliers, which can be used to improve numerical performances.

\paragraph{Acknowledgements:}
Fruitful discussions with H. Moulinec and J.-C. Michel are gratefully acknowledged. The Authors have received funding from Excellence Initiative of Aix-Marseille University - A*MIDEX, a French ``Investissements d'Avenir'' program in the framework of the Labex MEC.

%\clearpage
\appendix

\section{Mathematical definitions}

\subsection{Periodic fields}\label{math1}

%The following spaces of periodic functions will be useful.  
Consider a unit-cell $\mathcal{V}$ allowing to fill the space $\mathbb{R}^d$ by translation along $d$ vectors $\tou{Y}_1,\dots, \tou{Y}_d$. The lattice $\mathcal{R}$ generated by these vectors is defined as
\[
 \mathcal{R} = \left\{ \tou{Y},\;  \tou{Y} = \sum_{j=1}^d n_j \tou{Y}_j,\; n_j \in \mathbb{Z} \right\}.
 \]
Define spaces of periodic scalar functions, vector fields and tensor fields as:
\begin{eqnarray*}
& & L^2_\text{per}\!\left(\mathcal{V} \right)= \left\{ {f} \in L^2_\text{loc}(\mathbb{R}^d),\; f(\tou{x}+ \tou{Y}) =  f(\tou{x}), \; \text{a.e.} \ \tou{x} \in \mathcal{V}, \; \forall \tou{Y} \in \mathcal{R} \right\}, \\
& & H^1_\text{per}\!\left(\mathcal{V}\right)= \left\{ {f} \in H^1_\text{loc}(\mathbb{R}^d),\; f \in L^2_\text{per}\!\left(\mathcal{V}\right), \; \partial_{x_j} f \in L^2_\text{per}\!\left(\mathcal{V}\right), \; 1\leq j \leq d, \right\},\\ 
& & \tou{L}^2_\text{per}\!\left(\mathcal{V}\right)= \left\{ \tou{f}=({f}_j)\vert_{1\leq j \leq  d},\ f_j \ \in L^2_\text{per}\!\left(\mathcal{V}\right)\right\}, \\[1mm]
& & \tou{H}^1_\text{per}\!\left(\mathcal{V}\right)= \left\{ \tou{f}=({f}_j)\vert_{1\leq j \leq  d},\ f_j \ \in H^1_\text{per}\!\left(\mathcal{V}\right) \right\},\\[1mm]
& & \mathscrbf{L}^2_\text{per}\!\left(\mathcal{V}\right)= \left\{ \tod{F}=({F}_{jk})\vert_{ 1 \leq j,k \leq d},\ F_{jk}= F_{kj},\, F_{jk} \in L^2_\text{per}(\mathbb{R}^d)\right\}, %\\
\end{eqnarray*}

\subsection{Fourier transforms}\label{math2}

The Fourier transform $\hat{f}$ of $f$ is defined as: 
\[
\hat{f}(\tou{\xi}) = \TF[f](\tou{\xi}) =  \frac{1}{|\mathcal{V}|}  \int_ \mathcal{V} f( \tou{x} ) e^{-\mathrm{i} \tou{\xi} \cdot \tou{x}} \dd\tou{x},\quad \text{where}\quad   \mathrm{i} = \sqrt{-1}.
\]
Let $\mathcal{R}^*$ denote the reciprocal lattice of  $\mathcal{R}$ generated by the vectors
\[
 \tou{Y}^*_i = \frac{2 \pi}{|\mathcal{V}|} \tou{Y}_j \wedge \tou{Y}_k,
 \]
where $(i,j,k)$ is a direct circular permutation. Then, according to Plancherel's theorem:
\[
\frac{1}{|\mathcal{V}|} \int_\mathcal{V} \big| f(\tou{x})\big|^2 \ d \tou{x} =  \sum_{\tou{\xi} \in \mathcal{R}^*} \big|\hat{f}(\tou{\xi})\big|^2,
\]
and therefore
\[
f \in  L^2_\text{per}\!\left(\mathcal{V} \right) \quad \Leftrightarrow \quad  \sum_{\tou{\xi} \in \mathcal{R}^*} \big|\hat{f}(\tou{\xi})\big|^2 < + \infty.
\]
The original periodic function $f $ in $L^2_\text{per}\!\left(\mathcal{V} \right)$  can be reconstructed from its Fourier transform by
\[
 f(\tou{x}) = \TF^{-1}[\hat{f}](\tou{x}) = \sum_{\tou{\xi} \in \mathcal{R}^*} \hat{f}(\tou{\xi}) e^{\mathrm{i} \tou{\xi} \cdot\tou{x}}.
 \]
 
 \section{Properties of the Green's operators}\label{app:Green}
 
 \begin{lemma}\label{lemma:Green} The strain Green's operator $\toq{\Gamma}_0$ satisfies the following properties:
\begin{enumerate}
\item
$\toq{\Gamma}_0$ satisfies the reciprocity identity:
\begin{equation}
\pdse{\tod{\tau}}{\toq{\Gamma}_0 \tilde{\tod{\tau}}} = \pdse{ \tilde{\tod{\tau}}}{\toq{\Gamma}_0 \tod{\tau}}, \qquad \forall \tod{\tau}, \tilde{\tod{\tau}} \in \mathscrbf{H}_s.
\label{prop1}
\end{equation}
\item 
The kernel of $\toq{\Gamma}_0$ coincides with the subspace $\mathscrbf{S}$:
\begin{equation}
\toq{\Gamma}_0 \tod{s} = \tod{0} \quad \Leftrightarrow  \quad \tod{s}  \in \mathscrbf{S}. 
\label{prop3}
\end{equation}
\item 
$\toq{\Gamma}_0$ is such that:
\begin{equation}
\toq{\Gamma}_0 \toq{L}_0 \toq{\Gamma}_0  = \toq{\Gamma}_0.  
\label{prop3bis}
\end{equation}
\end{enumerate}
\end{lemma}
Properties \eqref{prop3} and \eqref{prop3bis} were proved in \cite{Michel}. The additional property \eqref{prop1} derives from the identity
\begin{equation}
\pdse{\tod{\tau}}{\toq{\Gamma}_0 \tilde{\tod{\tau}}} = \pse{\toq{\Gamma}_0 \tod{\tau}}{\toq{\Gamma}_0 \tilde{\tod{\tau}}}.
\label{self1ter}
\end{equation}
To prove \eqref{self1ter}, note that, by definition of $\toq{\Gamma}_0$, a stress field $\tod{s}=\toq{L}_0 \tod{\Gamma}_0\tod{\tau}-\tod{\tau}$ in $\mathscrbf{S}$ can be associated with $\tod{\tau}$ through \eqref{Gamma0}. By Lemma \ref{Hill:orth}:
\[
0 = \pdse{\tod{s}}{\tod{\Gamma}_0\tilde{\tod{\tau}}} = \pdse{\toq{L}_0 \tod{\Gamma}_0\tod{\tau}}{\tod{\Gamma}_0\tilde{\tod{\tau}}} - \pdse{\tod{\tau}}{\tod{\Gamma}_0\tilde{\tod{\tau}}},
\]
which, according to the definition \eqref{Riesz} of the Riesz mapping, proves \eqref{self1ter}. Note that similar properties can be proved for the stress Green's operator $\toq{\Delta}_0$ owing to the duality \emph{principle}, see \cite{Milton}.

\begin{remark}
If one defines the Green's operator from $\mathscrbf{L}^2_\text{per}\!\left(\mathcal{V}\right)$, endowed with the standard $L^2$-scalar product, into itself then the reciprocity identity \eqref{prop1} amounts in the self-adjointness of $\tod{\Gamma}_0$, a property which is known since \cite{Kroner}.  
\end{remark}

\end{document}